\documentclass[11pt, oneside]{amsart} 	              		
\usepackage{graphicx}				
\usepackage{amssymb, amsmath, amsthm, mathtools, xcolor, stmaryrd, enumitem,hyperref,tikz-cd, todonotes}

\title[GK bound for Banach representations with inf.~character]{Gelfand-Kirillov bound for $p$-adic Banach representations with infinitesimal character for $\GL_2$ and quaternion units}
\author{Reinier Sorgdrager}

\DeclareMathOperator{\orb}{orb}
\DeclareMathOperator{\Lie}{Lie}
\DeclareMathOperator{\GL}{GL}
\DeclareMathOperator{\Gal}{Gal}
\DeclareMathOperator{\Ext}{Ext}
\DeclareMathOperator{\Frob}{Frob}
\DeclareMathOperator{\id}{id}
\DeclareMathOperator{\Ann}{Ann}

\DeclareMathOperator{\Hom}{Hom}
\DeclareMathOperator{\Span}{Span}

\DeclareMathOperator{\Spec}{Spec}
\DeclareMathOperator{\SpecMax}{SpecMax}

\DeclareMathOperator{\Ind}{Ind}

\DeclareMathOperator{\Kdim}{K.dim}

\DeclareMathOperator{\gl}{\mathfrak{gl}}
\DeclareMathOperator{\gr}{gr}
\def\Fil{\mathop{\mathrm{Fil}}\nolimits}
\newtheorem{proposition}{Proposition}[section]
\newtheorem{corollary}{Corollary}[section]
\newtheorem{remark}{Remark}[section]
\newtheorem{lemma}{Lemma}[section]
\newtheorem{theorem}{Theorem}[section]

\newtheorem*{theorem*}{Theorem}
\theoremstyle{definition}
\newtheorem{definition}{Definition}[section]
\begin{document}
	\maketitle
	\begin{abstract}
		We prove that an admissible $p$-adic Banach representation of $\GL_2K$ whose locally analytic vectors have an infinitesimal character has Gelfand-Kirillov dimension $\leq[K\colon\mathbf Q_p]$, where $p>2$ and $K$ is a $p$-adic field. We also prove this for the group of units of the quaternions over $K$ replacing $\GL_2K$. In the process, we make some observations in the theory of $p$-adic Banach representations that might be of independent interest.
	\end{abstract}
	\tableofcontents
	\section{Introduction}
	The Gelfand-Kirillov dimension\footnote{Often going by the name ``canonical dimension'' but not in this article.} is an invariant of $p$-adic representations of $p$-adic Lie groups, such as smooth mod $p$ representations, $p$-adic Banach representations, or locally analytic representations.\\
	The invariant has received significant attention in the $p$-adic Langlands program, notably so in the work \cite{BHHMS1}. There it is proven, under sufficient assumptions, that the smooth mod $p$ representation $\pi$ of $\GL_2\mathbf Q_{p^f}$ associated to a $2$-dimensional mod $p$ Galois representation $\overline\rho$ of the absolute Galois group of $\mathbf Q_{p^f}$ via the geometry of Shimura curves has Gelfand-Kirillov dimension equal to $f$.\\
	Their statement has significant consequences, many of them technical and explored by the same authors in subsequent works. Let us only single out a rather immediate application, which is the non-zeroness of candidates for a conjectural $p$-adic Langlands correspondence for $\GL_2\mathbf Q_{p^f}$ as proposed by \cite{Caraiani_Emerton_Gee_Geraghty_Paškūnas_Shin_2018} via patching.\\ Namely, there is a patched module $\mathbb M_\infty$ over a complete local ring $R_\infty$ that is formally smooth over the framed Galois deformation ring\footnote{One actually also has to fix central characters in the setup of \cite{BHHMS1}, but let us ignore about this technical fact here.} $R_{\overline\rho}^\square$ of $\overline\rho$. This module admits an action of $\GL_2\mathbf Q_{p^f}$ and its special fiber is the dual $\pi^\vee$ of the mod $p$ representation $\pi$. If $y\in\SpecMax R_\infty[1/p]$, then it lies above a maximal ideal $x$ of $R_{\overline\rho}^\square[1/p]$ (that is a lift $\rho_x$ of $\overline\rho$ to a $p$-adic field) and $\mathbb M_\infty[1/p]/y$ is the dual of a $p$-adic Banach representation $\Pi_y$ of $\GL_2\mathbf Q_{p^f}$. In \cite{Caraiani_Emerton_Gee_Geraghty_Paškūnas_Shin_2018} it is conjectured that $\Pi_y$ depends in fact only depends on $x$ and in this way should correspond to $\rho_x$ via a conjectural $p$-adic Langlands correspondence for $\GL_2\mathbf Q_{p^f}$. Problematically, proving $\Pi_y$ is determined only by $x$ is very hard and a priori it is not even clear whether $\Pi_y\neq0$.\\
	However, \cite{BHHMS1} managed to prove at least that $\Pi_y\neq0$ for all $y$: the non-zeroness follows immediately from faithful flatness of $\mathbb M_\infty$ over $R_\infty$, which is a consequence of the Gelfand-Kirillov dimension being $f$ via ``Miracle Flatness'', as already observed by \cite{Gee_Newton_2022}. In fact, via the results from the appendix of \cite{Gee_Newton_2022} it is relatively straightforward to bound the Gelfand-Kirillov dimension of $\pi$ from below by $f$. The main contribution of \cite{BHHMS1} lies therefore in providing the upper bound, which is considered to be the hard part. Our work is also concerned with an upper bound, but in the Banach setting.
	\subsection*{Main theorem}
	In the context of $p$-adic Banach representations, \cite{dospinescu2023gelfandkirillovdimensionpadicjacquetlanglands} provide general bounds for the Gelfand-Kirillov dimension, when these admit an infinitesimal character: if $G$ is a connected reductive group over $\mathbf Q_p$ and $d_G$ is the dimension of the flag variety of $G_{\overline{\mathbf Q}_p}$, then one of their main results is that the Gelfand-Kirillov dimension of an admissible $p$-adic Banach representation of $G(\mathbf Q_p)$ whose locally analytic vectors admit an infinitesimal character is $<2d_G$.\\
	In \cite[Rem.~6.2]{dospinescu2023gelfandkirillovdimensionpadicjacquetlanglands} they raise the question whether one could prove the stronger bound $\leq d_G$. The main result of this paper is an affirmative answer to this question, in the case of $\GL_2$ or quaternion units:
		\begin{theorem*}[{Theorem \ref{mainmainthm}}]
		Let $p>2$ and let $K$ and $L$ be $p$-adic fields. Let $\Pi$ be an admissible $L$-Banach space representation of either $\GL_2K$ or of $D^\times$, where $D$ are the quaternions over $K$. If the locally analytic vectors of $\Pi$ have an infinitesimal character, then $\Pi$ has Gelfand-Kirillov dimension $\leq[K\colon\mathbf Q_p]$.
	\end{theorem*}
	At least in the case of $\GL_2$ this is the optimal bound, as equality is achieved by principal series representations, as we show in Section \ref{princseries}.\\ Our approach is rather different from that of \cite{dospinescu2023gelfandkirillovdimensionpadicjacquetlanglands} and bears similarity with \cite{BHHMS1} in a sense we discuss below. We have no reason to believe that the desired bound does not hold for more general groups and we plan to address this question in a future work. Indeed, preliminary computations for $\GL_3$ already suggest that our methods should yield an optimal bound in that respective setting.
	\subsection*{Associated gradeds and ``Casimir ideals''}
	The bound of \cite{BHHMS1} requires intricate methods, among which a very careful study of several Galois deformation rings, but it ultimately follows from a certain associated graded of $\pi^\vee$ being killed by a certain ideal. Namely, if $k_L$ is the field of mod $p$ coefficients of $\pi$ say, then $\pi^\vee$ is a finitely generated $k_L\llbracket I_1\rrbracket$-module, where $I_1$ is the (upper-triangular) pro-$p$ Iwahori subgroup of $\GL_2\mathbf Q_{p^f}$; the ring $k_L\llbracket I_1\rrbracket$ is local with maximal ideal $\mathfrak m$ and one can take the associated graded $\gr\pi^\vee$ with respect to the $\mathfrak m$-adic filtration to obtain a $\gr k_L\llbracket I_1\rrbracket$-module. The Gelfand-Kirillov dimension is (roughly) the Krull dimension of $\gr\pi^\vee$ and \cite{BHHMS1} manage to bound it by showing this module is annihilated by a certain ideal of $\gr k_L\llbracket I_1\rrbracket$ that only depends on $\GL_2\mathbf Q_{p^f}$ (and not on $\overline\rho$ for instance).\\
	Our bound is obtained in a very similar way. We start by picking a suitable open subgroup $H$ of $\GL_2K$ (or $D^\times$) that does the job of $I_1$ before and that is equipped with a $p$-valuation coming from the pro-$p$ Iwahori. Then we consider a finitely generated sub-$\mathcal O_L\llbracket H\rrbracket$-module $M_\circ$ inside the dual $\Pi^*$, where $\Pi$ is as in our main theorem, such that $M_\circ[1/p]=\Pi^*$. There are several filtrations we consider on $M_\circ$ and $\Pi^*$, as in \cite{cite-key}, coming from the $p$-valuation on $H$ and a choice of ``radius of analyticity'', but they all become essentially the same on the reduction $\overline M$ of $M_\circ$ modulo a uniformizer $\varpi_L$ of $L$. The Gelfand-Kirillov dimension of $\Pi$ equals the Krull dimension of the $\gr k_L\llbracket H\rrbracket$-module $\gr\overline M$ and  we arrive at the bound by showing $\gr\overline M$ is killed by a power of a certain ideal of $\gr k_L\llbracket H\rrbracket$.\\
	Again, this ideal only depends on the group and not on any particular properties of $\Pi$. One might call it the ``Casimir ideal'', as its generators arise from elements of the center of the universal enveloping algebra of the Lie algebra of $H$ -- although perhaps not in the most straightforward manner. Specializing to $K=\mathbf Q_{p^f}$ and to $\GL_2$ it recovers the ideal of \cite{BHHMS1}. In this way, we first of all offer a more conceptual explanation of the ideal found in \cite{BHHMS1} that is not yet present in their work. Second of all, our work suggests a candidate for the ideal one might want to consider in order to generalize the work of \cite{BHHMS1} to $\GL_2K$, where $K$ now has ramification. As an example, we explicitly describe the ideal for $\GL_2\mathbf Q_p[\sqrt p]$ in Remark \ref{explideal}.\\
	Although we certainly hope that our method could be used in a more direct way towards generalizing the work of \cite{BHHMS1}, it is unclear to us if it will. Let us point out that even when each $\Pi_y$ (obtained as above from a patched module $\mathbb M_\infty$, but in the more general setting of $\GL_2K$) has an infinitesimal character by\cite[Thm.~9.27]{dospinescu2023gelfandkirillovdimensionpadicjacquetlanglands} and therefore Gelfand-Kirillov dimension at most $[K\colon\mathbf Q_p]$, it does not follow that the special fiber of the patched module itself has Gelfand-Kirillov dimension bounded by this degree: this would only be the case if $\mathbb M_\infty/\left(R_\infty\cap y\right)$ has no $p$-torsion, as we point out in Remark \hyperref[patchremarka]{\ref*{patchremark} a)}. This torsion-freeness would certainly follow from (faithful) flatness of $\mathbb M_\infty$, but this in fact only shows that the statement of faithful flatness is equivalent to that of the Gelfand-Kirillov dimension of the special fiber being bounded from above by $[K\colon\mathbf Q_p]$. Whether there is any $y\in\SpecMax R_\infty[1/p]$ such that $\mathbb M_\infty/\left(R_\infty\cap y\right)$ is $p$-torsion free is not clear to us.\\
	Finally, it is worth remarking that \cite{dospinescu2023gelfandkirillovdimensionpadicjacquetlanglands} actually prove that $\mathbb M_\infty$ has an ``infinitesimal character in families''. While our method does not immediately extend to this setting, we hope to adapt it to the patched setting in the future.
	\subsection{Overview}
	Let us indicate the structure of the paper and at the same time reveal to some extent the ideas that go into the proof of the main result.
	The first sections are rather general in nature, and hold for any $p$-adic Lie group $H$ with a $p$-valuation.\\
	
	In \S\ref{sectionFiltrations} we mostly recall the constructions around the distribution algebras $D_r(H,L)$ and their filtrations, where $r\in[1/p,1)$, as developed in \cite{cite-key}. These induce filtrations on finitely generated modules over $\mathcal O_L\llbracket H\rrbracket$, $L\llbracket H\rrbracket\coloneqq\mathcal O_L\llbracket H\rrbracket[1/p]$, $D_r(H,L)$, or $k_L\llbracket H\rrbracket$. We observe that the associated graded $\gr_r\overline M$ does not depend on $r$ when $\overline M$ is a $k_L\llbracket H\rrbracket$-module. We also show that, for $M_\circ$ a finitely generated $\mathcal O_L\llbracket H\rrbracket$-module with reduction to $k_L$ denoted $\overline M$, the module $\gr_rM_\circ/\varepsilon_L$ ($\varepsilon_L$ is a principal symbol of a uniformizer of $L$) equals $\gr\overline M$ as soon as $\gr_rM_\circ$ is $\varepsilon_L$-torsion free.\\
		At the same time, $\gr_rM_\circ[1/\varepsilon_L]=\gr_r\left(D_r(H,L)\otimes_{\mathcal O_L\llbracket H\rrbracket}M_\circ\right)$ always holds.\\
		
		In \S\ref{sectionLie} we discuss the basics of the Lie theory $p$-adic Lie groups and its consequences for $p$-adic Banach and locally analytic representations. We explain the notion of infinitesimal character and most importantly we show how the universal enveloping algebra can be seen inside the distribution algebras. Via Lazard's isomorphism this allows us to view $\gr_{1/p}\mathcal O_L\llbracket H\rrbracket$ as a universal enveloping algebra.\\
		
		 In \S\ref{ppowersect} we study the distribution algebras and filtrations for the groups $H^{p^N}$ of $p$-power powers of $H$. Namely, the sets of $p^N$-th powers $H^{p^N}$ are open subgroups of $H$ when the $p$-valuation on $H$ is saturated. We introduce a notion that we call ``strictly saturated'', which is just slightly stronger than that of a $p$-valuation being saturated, and under this condition the filtrations for the different groups $H^{p^N}$ are nicely related. For example, $\gr_{1/p}\mathcal O_L\llbracket H^{p^N}\rrbracket$ (constructed with respect to a naturally induced $p$-valuation on $H^{p^N}$) will be a subring of $\gr_{r_N}\mathcal O_L\llbracket H\rrbracket$, where $r_N\coloneqq p^{-1/p^N}$. Consequently, many elements of $\gr_{r_N}\mathcal O_L\llbracket H\rrbracket$ now also get Lie theoretic interpretations.\\
		
		 In \S\ref{torfreesect} we prove that $\gr_{r_N}M_\circ$ has no $\varepsilon_L$-torsion provided that $M_\circ$ is $p$-torsion free and $N$ is sufficiently large. This generalizes a result of \cite{cite-key} who did this when $M_\circ$ has a single generator and our proof is a straightforward generalization of theirs.\\
		It has as corollary that $\Pi^\text{$r$-la}\subset\Pi$ is dense for $r$ sufficiently close to $1$, for admissible $p$-adic Banach representations $\Pi$ of any $p$-adic Lie group. The significance to us, however, is that in combination with the observations made in \S\ref{sectionFiltrations} it puts us in the following situation, for $N$ big enough:
		\begin{equation}\label{maindiagram}
			\gr\overline M=\gr_{r_N}M_\circ/\varepsilon_L\twoheadleftarrow\gr_{r_N}M_\circ\hookrightarrow\gr_{r_N}M_\circ[1/\varepsilon_L]=\gr_{r_N}\left(D_{r_N}(H,L)\otimes_{\mathcal O_L\llbracket H\rrbracket}M_\circ\right).
		\end{equation}
		In this way information coming from the infinitesimal action, which can be seen on $D_{r_N}(H,L)\otimes_{\mathcal O_L\llbracket H\rrbracket}M_\circ$, can be transported all the way to $\gr\overline M$. The proof of the main theorem works in this way: the presence of an infinitesimal character provides us with elements killing the $\gr_{r_N}\left(D_{r_N}(H,L)\otimes_{\mathcal O_L\llbracket H\rrbracket}M_\circ\right)$, from which we obtain elements killing $\gr\overline M$. Collecting these for several $N$ will provide us the ``Casimir ideal''. A single $N$ is not enough and it is therefore rather crucial that $\gr_{r_N}\overline M$ is actually independent of $N$.\\
		
		It is in \S\ref{sectionAppli} that we specialize to $\GL_2$ or quaternion units, but only after a brief recap of the notion of Gelfand-Kirillov dimension.\\
		After this, we define our setup, which consists of choosing a suitable subgroup $H$ with a suitable (strictly saturated) $p$-valuation for example. Our argument then uses the diagram (\ref{maindiagram}) extensively -- although we do not present the diagram as such in the main text. As such, the argument could actually be carried out with respect to any $H^{p^N}$. In particular, one can take $H^{p^N}$ to be uniform. However, we want to stress (as in Remark \ref{notuniform}) that it is very important that one does not take a corresponding ``uniform'' $p$-valuation that sends every element of an ordered basis to $1$.\\ Namely, our choice of $p$-valuation has the crucial consequence of rendering the Casimir elements homogenous (in a certain sense) so that their principal symbols in the various $\gr_{r_N}\mathcal O_L\llbracket H\rrbracket$ become sufficiently nice: they are polynomial expressions in the variable $\varepsilon_K$ with coefficients in $\gr k_L\llbracket H\rrbracket$; these coefficients will generate the Casimir ideal, and the homogeneity ensures that enough of these coefficients are non-zero as to sufficiently bound the Gelfand-Kirillov dimension.
	\subsection{Acknowledgements}
	The author thanks his advisor Benjamin Schraen, for proposing to investigate what annihilation properties one can obtain from Casimir elements, as well as for many fruitful discussions. He also thanks Christophe Breuil, Stefano Morra, and Vytautas Paškūnas for answering questions.
	\subsection{Preliminaries (mostly about filtrations)}
	Throughout this paper, we consider many filtered rings. The filtrations we consider are always decreasing indexed by $\mathbf R$ and exhaustive.
	If $R$ is a filtered ring with filtration $\Fil^\bullet R$, then a \emph{filtered $R$-module}, is an $R$-module $M$ with a filtration $\Fil^\bullet R$ of $R$-submodules of $M$ such that their union is $M$ and such that \begin{equation*}
		\Fil^tR\cdot\Fil^sM\subset\Fil^{s+t}M
	\end{equation*} for all $s,t\in\mathbf R$. The definition of a filtered ring is that it is ring that is a filtered module over itself with $1\in\Fil^0R$.\\
	If $R$ is a filtered ring (we omit the filtration in the notation) and $S$ another filtered ring which is also an $R$-algebra, then we say $S$ is a \emph{filtered $R$-algebra} if it is a filtered $R$-module.\\
	A very important construction we can do on a filtered $R$-module $M$ is taking its \emph{associated graded module} $\gr M$:\begin{equation*}
		\gr M=\bigoplus_{s\in\mathbf R}\Fil^sM/\Fil^{>s}M,
	\end{equation*} where $\Fil^{>s}M\coloneqq\cup_{t>s}\Fil^tM$. It is naturally a graded $\gr R$-module. Morphisms of filtered $R$-modules, which have an obvious definition, induce morphisms of associated gradeds.\\ If $x\in M$, then we denote by $\sigma(x)\in\gr M$ its \emph{principal symbol} -- if it exists -- which is a non-zero homogenous element with degree equal to $\sup\left\{s\colon x\in\Fil^sM\right\}$ and equal to the image of $x$ in the corresponding graded part. We will often call this degree the ``valuation'' of $x$ (with respect to the filtration on $M$). We stress that the operation of taking principal symbols is not at all additive: $\sigma(x)+\sigma(y)=\sigma(x+y)$ if and only if $\sigma(x)+\sigma(y)$ is homogenous and non-zero. When deemed necessary, we specify the filtration in the notation for $\gr M$ and the principal symbols.\\
	If $A$ and $B$ are both filtered $R$-modules, with filtrations $\Fil_A^\bullet$ and $\Fil_B^\bullet$ respectively, then $A\oplus B$ is naturally a filtered $R$-module with filtration defined by \begin{equation*}
		\Fil_{A\oplus B}^s\coloneqq\Fil_A^s\oplus\Fil_B^s.
	\end{equation*}
	We call this filtration the \emph{direct sum filtration}.
	If $A$ is a filtered $R$-algebra, then $A\otimes_RB$ is a filtered $A$-module via \begin{equation*}
		\Fil^s_{A\otimes_RB}\coloneqq\left\{\sum_ia_i\otimes b_i\colon\text{for each $i$ there is an $s_i\in\mathbf R$ such that }a_i\in\Fil_A^{s-s_i},b\in\Fil_B^{s_i}\right\}.
	\end{equation*} We call this filtration the \emph{tensor filtration}.\\
	If $F$ is a $p$-adic (skew) field, then the $p$-adic valuation $v_p$, which we always index so that $v_p(p)=1$, makes $F$ and its ring of integers $\mathcal O_F$ into filtered $\mathbf Z_p$-algebras. Our filtered rings of interest will be filtered $\mathbf Z_p$-algebras.\\
	We denote by $\varepsilon\in\gr\mathbf Z_p$ the principal symbol $\sigma(p)$ of $p$. When $\varpi_F$ is a fixed uniformizer of $F$, we also use the notation $\varepsilon_F$ for its principal symbol, which is of degree $1/e_F$, the reciprocal of the ramification index of $F$ over $\mathbf Q_p$.\\
	We have $\gr\mathcal O_F=k_F[\varepsilon_F]$ and $\gr F=k_F[\varepsilon_F^{\pm1}]$, denoting by $k_F$ the residue field of $F$.\\ 
	Finally, we will consider the following two ways of inducing filtrations:
	\begin{definition}
		Let $A\overset\iota\to B\overset f\to C$ be morphisms of sets and suppose $B$ has a filtration $\Fil^\bullet$. Then we define the following filtrations on $A$ and $C$ respectively:\begin{itemize}
			\item $\iota^*\Fil^s\coloneqq\{x\in A\colon\iota(x)\in\Fil^s\}$,
			\item $f_*\Fil^s\coloneqq\{x\in C\colon\exists b\in\Fil^s, f(b)=x\}$.
		\end{itemize}
	\end{definition}
	Throughout this text $L$ is a $p$-adic field functioning as our field of coefficients. From Section \ref{setupappl} on it will be subject to the assumption of it being sufficiently large in a suitable sense.\\
	
	Finally, let us  introduce the last bit of notation, which is unrelated to filtrations. We write $\Gal_F$ for the absolute Galois group of a $p$-adic field and we always denote by $\Frob\in\Gal(\mathbf Q_p^\text{ur}/\mathbf Q_p)$ the absolute arithmetic Frobenius automorphism of the maximal unramified field extension $\mathbf Q_p^\text{ur}$ of $\mathbf Q_p$.
	\section{Rings and filtrations}\label{sectionFiltrations}
	\subsection{Brief reminder on Banach representations}
	Let $G$ be a $p$-adic Lie group. Let $\Pi$ be an $L$-Banach representation of $G$, that is: $\Pi$ is an $L$-Banach space admitting a linear and continuous action of $G$. An example of an $L$-Banach representation is $C(G,L)$, the set of continuous functions from $G$ to $L$ with the sup-norm, under the condition that $G$ is compact. \\The locally analytic vectors of $\Pi$ are those $v\in\Pi$ for which the orbit map \begin{equation*}
		\orb_v\colon G\to\Pi, g\mapsto gv
	\end{equation*} is locally analytic, which means that locally around each $g\in G$ one can find a chart $\varphi\colon B_\varepsilon(0)\to U\subset G$ mapping $0$ to $g$ such that $\orb_v\circ\varphi$ is given by a power series \begin{equation*}
		\sum_{\alpha\in\mathbf Z_{\geq0}^d}x_1^{\alpha_1}\cdot\cdots\cdot x_d^{\alpha_d}v_\alpha,\ \ \ \ \ \ \ \ (x_1,\dots,x_d)\in B_\varepsilon(0)\subset\mathbf Q_p^d
	\end{equation*} satisfying $\lim_{|\alpha|\to\infty}\varepsilon^{|\alpha|}\|v_\alpha\|=0$, where $|\alpha|\coloneqq\sum_i\alpha_i$.\\ The set $\Pi^\text{la}$ of all locally analytic vectors of $\Pi$ is a representation of $G$. One can derive the action of $G$ on $\Pi^\text{la}$, which leads to the infinitesimal action of the Lie algebra of $G$ on $\Pi^\text{la}$ and to the notion of infinitesimal character, which we detail more in Section \ref{sectionLie}. In our example, we have $C(G,L)^\text{la}=C^\text{la}(G,L)$, the latter being the set of locally analytic functions $G\to L$.\\
	
	We are mostly interested in \emph{admissible} $L$-Banach representations, that is, those $\Pi$ that embed in $C(H,L)^{\oplus N_0}$ as $L$-Banach representation of $H$, for some compact open subgroup $H$ of $G$ and some integer $N_0$. If $\Pi$ is admissible, one can find such integer $N_0$ for any open compact subgroup $H$. Admissible representations lend themselves well to methods of non-commutative algebra, because it turns out that the continuous linear dual \begin{equation}\label{dual}
		(-)^*\coloneqq\Hom_L^\text{cont}\left(-,L\right)
	\end{equation}
	induces an anti-equivalence from the category of admissible $L$-Banach representations of $G$ to the category of $L\llbracket G\rrbracket$-modules that are finitely generated over $L\llbracket H\rrbracket$ for some (equivalently any) compact open subgroup $H$ of $G$. Here $L\llbracket G\rrbracket$ is the so-called \emph{augmented} Iwasawa algebra or completed group ring $L\llbracket G\rrbracket\coloneqq L[G]\otimes_{L[H]}L\llbracket H\rrbracket$ ($H$ being any open compact subgroup of $G$ -- therefore profinite -- again), where $L\llbracket H\rrbracket\coloneqq\mathcal O_L\llbracket H\rrbracket[1/p]$.\\
	In general, with $G$ not necessarily compact, we give $C(G,L)$ the \emph{strong topology} or equivalently the \emph{topology of bounded convergence}, which means that it is a topological vector space with basis of open neighbourhoods around $0$ given by all subsets $\{f\in C(G,L)\colon f(B)\subset\varpi_L^n\mathcal O_L\}$, where $B\subset G$ is compact and $n\in\mathbf Z$. Then we can make sense of $C(G,L)^*$ and, giving $C^\text{la}(G,L)$ the subspace topology, also of $C^\text{la}(G,L)^*$. The first one we already know, since one can check that $L\llbracket G\rrbracket=C(G,L)^*$. The second is naturally an $L$-algebra known as the algebra of locally analytic distributions on $G$ and is written $D(G,L)\coloneqq C^\text{la}(G,L)^*$.\\
	The group $H$ can be naturally seen as (multiplicative) subset of the distribution algebras $D(H,L)$ and $L\llbracket H\rrbracket$, by identifying an element $h\in H$ with the Dirac distribution it defines. For purposes of clarity, for $h\in H$, we will always denote it between brackets when considered as element of $\mathcal O_KL\llbracket H\rrbracket$ (or one of the larger algebras): by $[h]\in\mathcal O_L\llbracket H\rrbracket$.\\
	
	When $\Pi$ is admissible, it turns out that naturally \begin{equation}
		\Pi^{\text{la},*}=D(G,L)\otimes_{L\llbracket G\rrbracket}\Pi^*,
	\end{equation} so that the dual of $\Pi^\text{la}$ is a finitely generated $D(G,L)$-module and even a finitely generated $D(H,L)$-module, for any compact open subgroup $H$ of $G$ (this is the notion of \emph{strongly admissible}). Moreover, $\Pi^\text{la}$ is dense in $\Pi$. See \cite[Thm.~7.1]{cite-key}.\\
	The continuous linear dual $\Pi^{\text{la},*}$ captures all the information of $\Pi^\text{la}$ via an anti-equivalence between the category of strongly admissible locally analytic $L$-representations of $G$ to the category of $D(G,L)$-modules that are finitely generated over some (equivalently any) $D(H,L)$, for $H$ a compact open subgroup of $G$, see \cite[Cor.~3.3]{ST02}.\\
	
	Our main object of study in this paper will be admissible $L$-Banach representations $\Pi$. For the above reasons we will however prefer to work with finitely generated $L\llbracket H\rrbracket$-modules or even finitely generated $\mathcal O_L\llbracket H\rrbracket$, for $H$ a compact open subgroup. In fact we like to specialize our situation even more, requiring that $H$ is a $p$-valued group, which allows us to put useful filtrations on the various rings constructed from $H$. We discuss this next.
	\subsection{$p$-valued groups and their associated filtered rings}\label{filtrationsdefsection}
	A \emph{$p$-valuation} on a group $H$ is a map $\omega\colon H\to(\frac1{p-1},\infty]$ satisfying
	\begin{itemize}
		\item$\omega(xy^{-1})\geq\min\{\omega(x),\omega(y)\}$,
		\item$\omega([x,y])\geq\omega(x)+\omega(y)$,
		\item$\omega(x)=\infty$ if and only if $x$ is trivial,
		\item$\omega(x^p)=\omega(x)+1$,
	\end{itemize} for all $x,y\in H$, where $[x,y]=xyx^{-1}y^{-1}$.\\
	A $p$-valuation naturally defines a topology on $H$ making it into a topological group. We call $\omega$ \emph{saturated} if $H$ is complete for the topology it defines and if $\omega(h)>p/(p-1)$ implies $h\in H^p\coloneqq\{x^p\colon x\in H\}$.\\
	We shall only work with $p$-valued $H$ that are also Lie groups. This remains rather general, as every $p$-adic Lie group contains an open subgroup that admits a $p$-valuation and one can even demand the $p$-valuation to be saturated and integer-valued, see \cite[Thm.~27.1]{Schneider2011}.\\ 
	In the case $H$ is also a $p$-adic Lie group, it will have an \emph{ordered basis} of elements $h_1,\dots, h_d\in H$, where $d$ is its dimension as a $p$-adic Lie group. This means that the map \begin{equation}\label{homeoordbas}
		\mathbf Z_p^d\to H, (a_1,\dots,a_d)\mapsto h_1^{a_1}h_2^{a_2}\cdots h^{a_d}, 
	\end{equation} is a homeomorphism and in facts gives a chart for $H$ as a $p$-adic manifold.\\
	Supposse for the remainder of this section that $H$ is a $p$-adic Lie group with a $p$-valuation $\omega$.\\
	
	 An ordered basis allows for explicit descriptions of the rings $\mathcal O_L\llbracket H\rrbracket$, $k_L\llbracket H\rrbracket$, $L\llbracket H\rrbracket$, and $D(H,L)$ in the following sense. We write $b_i\coloneqq [h_i]-1$ and $\mathbf b^\alpha\coloneqq b_1^{\alpha_1}\cdot\cdots\cdot b_d^{\alpha_d}$ for $\alpha=(\alpha_1,\dots,\alpha_d)\in\mathbf Z_{\geq0}^d$. Then, first of all, we have  \begin{equation*}
		\mathcal O_L\llbracket H\rrbracket=\left\{\sum_\alpha c_\alpha\mathbf b^\alpha\colon c_\alpha\in\mathcal O_L\right\}
	\end{equation*} and each element of the completed group ring is expressed as such a power series in a unique way. The elements of $L\llbracket H\rrbracket\coloneqq\mathcal O_L\llbracket H\rrbracket[1/p]$ can therefore also be expressed uniquely by power series as do the elements of $k_L\llbracket H\rrbracket$.\\
	The homeomorphism (\ref{homeoordbas}) is also very useful as the theory of Mahler expansions allows one to then write \begin{equation*}
		D(H,L)\coloneqq\left\{\sum_\alpha c_\alpha\mathbf b^\alpha\colon c_\alpha\in L, |c_{\alpha}|r^{|\alpha|}\to0\text{ as $|\alpha|\to\infty$, for all } r\in(0,1)\right\},
	\end{equation*} see \cite[\S4]{cite-key}.\\
	Moreover, the $p$-valuation $\omega$ gives us several norms on these algebras. Write $\tau\alpha\coloneqq\sum_i\alpha_i\omega(h_i)$. For $r\in[1/p,1)$ we define the norm $\|-\|_r$ on $D(H,L)$ given by \begin{equation*}
		\|\sum_\alpha c_\alpha\mathbf b^\alpha\|_r\coloneqq\sup_\alpha|c_\alpha|r^{\tau\alpha}.
	\end{equation*}
	We remark that this norm only depends on $r$ and $\omega$ and not on the choice of ordered basis, see the discussion after Theorem 4.10 of \cite{cite-key}.\\
	Completing with respect to this norm we obtain the ``radius $r$ distribution algebra'' \begin{equation*}
		D_r(H,L)=\left\{\sum_\alpha c_\alpha\mathbf b^\alpha\colon c_\alpha\in L, |c_{\alpha}|r^{\tau\alpha}\to0\text{ as }|\alpha|\to\infty\right\}.
	\end{equation*}
	This norm induces a filtration on $D_r(H,L)$ which we index additively in order to obtain a filtered $\mathbf Z_p$-algebra: \begin{equation}\label{filtronDr}
		\Fil_r^sD_r(H,L)=\left\{\sum_\alpha c_\alpha\mathbf b^\alpha\colon v_p(c_{\alpha})-\frac{\log r}{\log p}\cdot\tau\alpha\geq s\right\}.
	\end{equation}
	Via \begin{equation}\label{rings}
		k_L\llbracket H\rrbracket\twoheadleftarrow\mathcal O_L\llbracket H\rrbracket\subset L\llbracket H\rrbracket\subset D_r(H,L)
	\end{equation} we obtain unambiguous induced filtrations on all these rings $R$, which we denote by $\Fil_r^\bullet R$. We denote the associated graded rings by $\gr_rR$ and the principal symbols by $\sigma_r(-)$ in all cases except that of $k_L\llbracket H\rrbracket$ when we use $\overline\sigma_r(-)$.
	\begin{remark}
		Let us warn the reader that even while the notation is perhaps suggestive and one has that $\gr_rk_L\llbracket H\rrbracket=\gr_r\mathcal O_L\llbracket H\rrbracket/\varepsilon_L$ it is not the case that the reduction modulo $\varepsilon_L$ of $\sigma_r(x)$, where $x\in\mathcal O_L\llbracket H\rrbracket$, necessarily equals $\overline\sigma_r(\overline x)$.
	\end{remark}
	Finally, let us mention that if $x=\sum_\alpha c_\alpha\mathbf b^\alpha\in D_r(H,L)\setminus\{0\}$, then the degree of $\sigma_r(x)$ in $\gr_rD_r(H,L)$ is equal to $v_r(x)\coloneqq\inf_\alpha v_p(c_\alpha)-\frac{\log r}{\log p}\cdot\tau\alpha$, which is the additive ``$r$-valuation'' version of $\|-\|.$ 
	\subsection{Constructions from finite type $\mathcal O_L\llbracket H\rrbracket$-modules}\label{constrsect}
	Let $M_\circ$ be a finite type $\mathcal O_L\llbracket H\rrbracket$-module. For example, if $\Pi$ is an admissible $L$-Banach representation of $H$ and $\Pi_\circ$ is an $H$-invariant unit ball in $\Pi$, then one could take \begin{equation*}
		M_\circ\coloneqq\Hom_{\mathcal O_L}^\text{cont}\left(\Pi_\circ,\mathcal O_L\right)\subset\Hom_L^\text{cont}\left(\Pi,L\right)=M_\circ[1/p].
	\end{equation*}
	The aim of this section is to make the modules $M_\circ$ and $M\coloneqq M_\circ[1/p]$ into filtered $\mathcal O_L\llbracket H\rrbracket$- or $L\llbracket H\rrbracket$-modules, respectively, for the $\Fil_r^\bullet$-filtrations on these rings for any $r\in[1/p,1)$. Likewise, we will do so for the $M_r\coloneqq D_r(H,L)\otimes_{\mathcal O_L\llbracket H\rrbracket}M_\circ$ and $\overline M\coloneqq k_L\llbracket H\rrbracket\otimes_{\mathcal O_L\llbracket H\rrbracket}M_\circ$.\\
	We keep the above notation for the modules obtained from $M_\circ$ throughout this text.\\
	
	Fix a surjection \begin{equation}\label{presentation}
		\mathcal O_L\llbracket H\rrbracket^{\oplus N_0}\twoheadrightarrow M_\circ. 
	\end{equation}
	For each $r\in[1/p,1)$ we have the filtration $\Fil_r^\bullet$ on $\mathcal O_L\llbracket H\rrbracket$ and therefore also on the direct sum. Via the surjection we obtain the induced filtration $\Fil_r^\bullet M_\circ$, whose notation abusively does not take the surjection into account.\\
	After tensoring (\ref{presentation}) by one of the $\mathcal O_L\llbracket H\rrbracket$-algebras $R$ of (\ref{rings}) we likewise get surjections $R^{\oplus N_0}\twoheadrightarrow R\otimes_{\mathcal O_L\llbracket H\rrbracket}M_\circ$ that induce filtrations $\Fil_r^\bullet\left(R\otimes_{\mathcal O_L\llbracket H\rrbracket}M_\circ\right)$, using the filtration $\Fil_r^\bullet R$ on $R$. We denote the associated graded modules (over the associated graded rings) by $\gr_rR\otimes_{\mathcal O_L\llbracket H\rrbracket}M_\circ$ and we also extend our notation for principal symbols above to this setting.\\
	
	The filtration $\Fil_r^\bullet$ defines a topology on $M$, and we denote by $M^{\wedge_r}$ its completion, which is a $D_r(H,L)$-module. By completing the $\Fil^s_rM$ as well, one obtains a filtration $\Fil_r^sM^{\wedge_r}$ that turns $M^{\wedge_r}$ into a filtered $D_r(H,L)$-module.
	\begin{lemma}
		The completion $M^{\wedge_r}$ is canonically isomorphic to $M_r$ as filtered $D_r(H,L)$-module.
	\end{lemma}
	\begin{proof}
		The natural map $M\to M_r$ is continuous for the topologies on both modules induced by their respective filtrations. Therefore it induces a morphism $f\colon M^{\wedge_r}\to M_r$, because $M_r$ is complete for its topology, which is a consequence of the completeness of $D_r(H,L)$. By the universal property of the tensor product we also have a morphism $g\colon M_r=D_r(H,L)\otimes_{L\llbracket H\rrbracket}M\to M^{\wedge_r}$, again induced by $M\to M_r$ and the $D_r(H,L)$-action on the completion. Both $M_r$ and $M^{\wedge_r}$ admit natural morphisms from $M$ and are generated over $D_r(H,L)$ by the image of $M$. The compositions $f\circ g$ and $g\circ f$ are the identity on the respective image of $M$ and are therefore equal to the identity of $M_r$ and $M^{\wedge_r}$, respectively.\\
		Finally, $f$ and $g$ respect the filtrations, because the natural maps $M\to M_r$ and $M\to M^{\wedge_r}$ are morphisms of filtered $L\llbracket H\rrbracket$-modules.
	\end{proof}
	The following lemma shows when different candidates for filtrations on our modules coincide.
	\begin{lemma}\label{commutation}
		\begin{enumerate}[label*=\roman*)]
			\item The two possible ways of inducing a filtration on $M$ from $\Fil_r^\bullet D_r(H,L)^{\oplus N_0}$ and via the diagram \begin{equation*}
				\begin{tikzcd}
					L\llbracket H\rrbracket^{\oplus N_0} \arrow[d, two heads, "f"] \arrow[r, hook, "\iota"] & {D_r(H,L)}^{\oplus N_0} \arrow[d, two heads, "f'"] \\
					M \arrow[r, "\iota'"]                                                        & M_r                            
				\end{tikzcd}
			\end{equation*} agree. That is, $f_*\iota^*\Fil_r^\bullet D_r(H,L)^{\oplus N_0}=\iota'^*f'_*\Fil_r^\bullet D_r(H,L)^{\oplus N_0}$.\\ Consequently, we have a canonical isomorphism $\gr_rM\cong\gr_rM_r$ of $\gr_rL\llbracket H\rrbracket$-modules.
			\item If $\gr_rM_\circ$ is $\varepsilon$-torsion free, then the same holds for \begin{equation*}
				\begin{tikzcd}
					\mathcal O_L\llbracket H\rrbracket^{\oplus N_0} \arrow[d, two heads, "f_\circ"] \arrow[r, hook, "\iota_\circ"] & {D_r(H,L)}^{\oplus N_0} \arrow[d, two heads, "f'"] \\
					M_\circ \arrow[r,"\iota_\circ'"]                                                        & M_r                            
				\end{tikzcd}.
			\end{equation*} That is, $f_{\circ,*}\iota_\circ^*\Fil_r^\bullet D_r(H,L)^{\oplus N_0}=\iota_\circ'^*f'_*\Fil_r^\bullet D_r(H,L)^{\oplus N_0}$.
		\end{enumerate}
	\end{lemma}
	\begin{proof}
		For i) we need to show that the filtration $\Fil_r^\bullet M=f_*\iota^*\Fil_r^\bullet D_r(H,L)^{\oplus N_0}$ is equal to $\iota'^*f'_*\Fil_r^\bullet D_r(H,L)^{\oplus N_0}$. By the preceding lemma, this latter filtration, is actually the filtration on $M$ induced by $M\to M^{\wedge_r}$ and $\Fil_r^\bullet M^{\wedge_r}$, which is indeed $\Fil^\bullet_rM$. The isomorphism $\gr_rM\cong\gr_rM_r$ is then a consequence of the image of $M$ being dense in $M_r$.\\
		
		Now we prove ii), so we assume the $\varepsilon$-torsion freeness. By i) we may equivalently proof the result for the diagram \begin{equation*}
			\begin{tikzcd}
				\mathcal O_L\llbracket H\rrbracket^{\oplus N_0} \arrow[d, two heads] \arrow[r, hook] & {L\llbracket H\rrbracket}^{\oplus N_0} \arrow[d, two heads] \\
				M_\circ \arrow[r, hook]                                                        & M                            
			\end{tikzcd}.
		\end{equation*}The $\varepsilon$-torsion freeness means that for all $s\in\mathbf R$ and all $i\in\mathbf Z_{\geq1}$\begin{equation*}
			\Fil_r^sM_\circ\cap \varpi_L^iM_\circ=\varpi_L\Fil^{s-i/e_L}_rM_\circ.
		\end{equation*} At the same time we have for each $s$\begin{equation*}
			\Fil_r^sM=\bigcup_{i\in\mathbf Z}\varpi_L^i\Fil_r^{s-i/e_L}M_\circ,
		\end{equation*} because the filtration on $L\llbracket H\rrbracket^{\oplus N_0}\cong L\otimes_{\mathcal O_L}\mathcal O_L\llbracket H\rrbracket^{\oplus N_0}$ is the tensor filtration.\\
		Combining the two equalities gives that \begin{equation*}
			M_\circ\cap\Fil_r^sM=\Fil_r^sM_\circ
		\end{equation*} from which we conclude.
	\end{proof}
	We also have the following:
	\begin{lemma}\label{localizationtoeps}
		The natural map $\gr_rM_\circ[1/\varepsilon]\to\gr_rM=\gr_rM_r$ is an isomorphism.
	\end{lemma}
	\begin{proof}
		First of all, this map is a surjection. To see this, take any non-zero $x\in\gr_rM$, which we may suppose homogenous of some degree $s\in\mathbf R$ without loss of generality. It is the principal symbol $\sigma_r(\tilde x)$ of some element $\tilde x$ in $M$ which we can write as $\frac1{p^n}\tilde x'$ with $\tilde x'$ in $\Fil_r^{s+n}M_\circ$ (because even if the filtration on $M$ is not defined that way, it coincides with the tensor filtration of $L\otimes_{\mathcal O_L}M_\circ$). The principal symbol $\sigma_r(\tilde x')\in\gr_rM_\circ$ exists and is in degree $n+s$, because if it were in higher degree, the degree of $\sigma_r(\tilde x)$ would necessarily be higher than $s$ which is a contradiction. We see that $\varepsilon^n\sigma_r(\tilde x')=x$ and we obtain the surjectivity.\\
		For injectivity take any non-zero $y\in\gr_rM_\circ[1/\varepsilon]$ that maps to zero in $\gr_rM$, homogenous of degree $s\in\mathbf R$ without loss of generality, and write it as $\sigma_r(\tilde y')/\varepsilon^n$. Since it maps to zero, apparently $\tilde y'\in\Fil_r^{>s+n}M$, so that there are $k\in\mathbf Z_{\geq1}$ and $\tilde z\in\Fil_r^{>s+n+k}M_\circ$ such that $\tilde y'=\tilde z/p^k$, using again that we are dealing with the tensor filtration $L\otimes_{\mathcal O_L}M_\circ$. As a result $\varepsilon^k\sigma_r(\tilde y')=0$ in $\gr_rM_\circ$ so that $y\in\gr_rM_\circ[1/\varepsilon]$ was actually zero after all, proving the injectivity by contradiction.
	\end{proof}
	\subsection{There can be no confusion about $\gr\overline M$}\label{modpsect}
	Keep the fixed presentation $\mathcal O_L\llbracket H\rrbracket^{\oplus N_0}\twoheadrightarrow M_\circ$. Then, as discussed above, each $\Fil_r^\bullet M_\circ$ induces a filtered module $\Fil_r^\bullet\overline M$ over the filtered ring $\Fil_r^\bullet k_L\llbracket H\rrbracket$ (the same as the one induced by $\Fil_r^\bullet k_L\llbracket H\rrbracket^{\oplus N_0}$ via $k_L\llbracket H\rrbracket^{\oplus N_0}\twoheadrightarrow\overline M$). Here we show that these are all essentially the same, that is, independent of the choice of radius $r$, up to rescaling the filtration indices.
	\begin{proposition}
		Let $r,r'\in[1/p,1)$. Then, for all $s\in\mathbf R$, $$\Fil_r^s\overline M=\Fil_{r'}^{\frac{\log r'}{\log r}\cdot s}\overline M.$$
	\end{proposition}
	\begin{proof}
		It suffices to prove the statement for $\overline M=k_L\llbracket H\rrbracket$, because the filtrations on $\overline M$ are all induced by the same presentation $k_L\llbracket H\rrbracket^{\oplus N_0}\twoheadrightarrow\overline M$.\\ 
		There, we have the following description\begin{equation*}
			\Fil_r^sk_L\llbracket H\rrbracket=\left\{\sum_\alpha c_\alpha\mathbf b^\alpha\colon-\frac{\log r}{\log p}\cdot\tau\alpha\geq s\right\},
		\end{equation*} easily deduced from (\ref{filtronDr}). Therefore, $$\Fil_r^sk_L\llbracket H\rrbracket=\Fil_{r'}^{\frac{\log r'}{\log r}\cdot s}k_L\llbracket H\rrbracket.$$
	\end{proof}
	Consequently, the rings $\gr_rk_L\llbracket H\rrbracket$, $r$ ranging over $[1/p,1)$, are all canonically isomorphic, and we will denote them simply by $\gr k_L\llbracket H\rrbracket$. Likewise, we write $\gr\overline M$ for the canonically isomorphic $\gr k_L\llbracket H\rrbracket$-modules $\gr_r\overline M$ (always keeping the presentation of $\overline M$ fixed) and $\overline\sigma(-)$ instead of $\overline\sigma_r(-)$ for the principal symbols.\\
	Finally, when $\gr_r M_\circ$ is $\varepsilon$-torsion, $\gr\overline M$ has yet another interpretation:
	\begin{lemma}\label{modeps=grmodp}
		If $\gr_rM_\circ$ is $\varepsilon$-torsion free, then $\gr\overline M=\gr_rM_\circ/\varepsilon_L\gr_rM_\circ$.
	\end{lemma}
	\begin{proof}
		By $\varepsilon$-torsion freeness we have for every $s\in\mathbf R$ that $$\Fil_r^sM_\circ\cap \varpi_LM_\circ=\varpi_L\Fil^{s-1/e_L}_rM_\circ.$$
		Therefore the second equality in the following chain is valid:
		\begin{gather*}
			\begin{aligned}
				\gr_rM_\circ/\varepsilon_L\gr_rM_\circ&=\bigoplus_s\Fil_r^sM_\circ/\left(\Fil_r^{>s}M_\circ+\varpi_L\Fil^{s-1/e_L}_rM_\circ\right)\\
				&=\bigoplus_s\Fil_r^sM_\circ/\left(\Fil_r^{>s}M_\circ+\Fil_r^sM_\circ\cap \varpi_LM_\circ\right)\\&\cong\bigoplus_s\left(\Fil_r^sM_\circ+\varpi_LM_\circ\right)/\left(\Fil_r^{>s}M_\circ+\varpi_LM_\circ\right)\\&=\gr_rM_\circ/\varpi_LM_\circ.
			\end{aligned}
		\end{gather*}
	\end{proof}	
	\section{Lie theory}\label{sectionLie}
	We start by recalling some basic Lie theory for $p$-adic Lie groups, which can be found in \cite[Ch.~III, VII]{Schneider2011}. In this section $G$ is still a $p$-adic Lie group, and $H\subset G$ an open $p$-valued subgroup.
	\subsection{Exponential map}
	As usual, the Lie algebra $\Lie G$ of $G$ is defined to be the tangent space of $G$ at the identity, and has the structure of a Lie algebra over $\mathbf Q_p$. There is the exponential map
	$\exp\colon\Lie G\dashrightarrow G$, but it is only defined on a small enough open of $\Lie G$, on which is it analytic. Nonetheless the situation is rather favorable: choosing a norm on $\Lie G$ we obtain open balls $G_\varepsilon\subset\Lie G$ of radius $\varepsilon>0$ around $0$ and for $\varepsilon$ sufficiently small these can be given a canonical structure of a $p$-adic Lie group. The exponential $\exp\colon\Lie G\dashrightarrow G$ is defined on these $G_\varepsilon$ and actually embeds $G_\varepsilon$ into $G$ as an open subgroup. The $\{G_\varepsilon\}_\varepsilon$ are called the \emph{Lie group germs} of the Lie algebra $\Lie G$. 
	\subsection{Infinitesimal action and character}
	Let $\Pi$ be and admissible $L$-Banach representation of $G$. Then an element $\mathfrak r\in\Lie G$ acts on the locally analytic vectors $\Pi^\text{la}$ by derivations via the exponential map:
	\begin{equation*}
		\mathfrak r.v\coloneqq\frac d{dt}\exp(t\mathfrak r)v\big|_{t=0}
	\end{equation*} for $v\in\Pi^\text{la}$.
	This equation makes sense, as $t\mathfrak r$ will lie in some $G_\varepsilon$ when $t$ is small enough.\\
	If $M$ is the continuous $L$-dual of $\Pi$, then the action of $\Lie G$ on the dual of $\Pi^\text{la}$, which is $D(H,L)\otimes_{L\llbracket H\rrbracket}M$, is given by a map of $\mathbf Q_p$-Lie algebras which we abusively describe as \begin{equation}\label{psi}
		\psi\colon\Lie G\to D(H,L), \mathfrak r\mapsto\log\left([\exp\mathfrak r]\right)=-\sum_{i\geq1}\frac{(1-[\exp\mathfrak r])^i}i.
	\end{equation}
	The abuse of notation here is twofold: first of all, $\exp\mathfrak r$ might simply not be defined, and second of all, even if it were defined, this does not mean $\exp\mathfrak r$ is necessarily an element of $H$. However, taking $N$ big enough, $\log\left([\exp p^N\mathfrak r]\right)$ will be an element of $D(H,L)$ and we define $\psi(\mathfrak r)$ to be $\frac1{p^N}\log\left([\exp p^N\mathfrak r]\right)$.\\
	The map naturally extends to the universal enveloping algebra \begin{equation}\label{infaction}
		U_L(L\otimes_{\mathbf Q_p}\Lie G)\to D(H,L).
	\end{equation}
	In these terms the infinitesimal character can be described, as we now briefly recall. Namely, let $\lambda\colon\mathcal Z(U_L(L\otimes_{\mathbf Q_p}\Lie G))\to L$ be a morphism of $L$-algebras, where $\mathcal Z(U_L(L\otimes_{\mathbf Q_p}\Lie G))$ denotes the center of $U_L(L\otimes_{\mathbf Q_p}\Lie G)$. We say $\Pi$ has \emph{infinitesimal character} $\lambda$ if the action of $U_L(L\otimes_{\mathbf Q_p}\Lie G)$ on $D(H,L)\otimes_{{L\llbracket H\rrbracket}}M$ is given by $\lambda$.
	\subsection{The $\mathbf Z_p$-Lie algebra $H$}
	Of course, going in the other direction we can also identify at least some Lie lattices of $\Lie G$ inside the group $G$, and it turns out that if $H$ is saturated then $H$ itself naturally has the structure of a Lie algebra. Namely, under the following addition and Lie bracket $H$ then becomes a $\mathbf Z_p$-Lie algebra, see \cite[exercice (III, 2.1.10)]{Lazard1965}\footnote{Lazard's exercise is not at all evident, as noted by Serre in his Bourbaki talk, \cite[\S4.2]{Serre1962-1964}. Serre briefly discusses there how one can prove it via the theory developed in \cite{Lazard1965}, which can also be found in \cite[Ch.~VII]{Schneider2011}. Alternatively, one solves the exercise by direct computation, in the spirit of \cite{ddms}, and the reader is advised to look up several useful identities in the latter book as well as in \cite{Schneider2011}.}:
	\begin{gather*}
		\begin{aligned}
			&g+h\coloneqq\lim_{n\to\infty}(g^{p^n}h^{p^n})^{p^{-n}},\\
			&[g,h]\coloneqq\lim_{n\to\infty}(g^{p^n}h^{p^n}g^{p^{-n}}h^{p^{-n}})^{p^{-2n}}.
		\end{aligned}
	\end{gather*}
	The ordered basis then becomes a basis for this as a $\mathbf Z_p$-module and the $p$-valuation $\omega$ on $H$ makes $H$ into a filtered $\mathbf Z_p$-module, where $\mathbf Z_p$ is considered with its $p$-adic filtration. Its associated graded $\gr H$ is then a Lie algebra over $\gr\mathbf Z_p=\mathbf F_p[\varepsilon]$, where the Lie bracket treats the degree additively.\\
	Let us note that $H$ actually appears as a Lie group germ of $\Lie G$ via a suitable norm on $\Lie G$, see \cite[\S31]{Schneider2011}. We call the corresponding identifying map the logarithm which we notate as $\log\colon H\to\Lie H=\Lie G$ and which is an embedding of $\mathbf Z_p$-algebras.\\ Although we will not use this, let us mention that the larger context of what we have seen here sofar is Lazard's equivalence \cite[(IV, 3.2.6)]{Lazard1965} between the category of saturated $p$-valued groups and the category of what \cite{Serre1962-1964} calls ``*-saturated'' $\mathbf Z_p$-Lie algebras.
	\subsection{Lazard's isomorphism}
	Assume $H$ with its $p$-valuation is saturated, so that $H$ can be seen as $\mathbf Z_p$-Lie algebra. Consider the universal enveloping algebra $U_{\mathbf Z_p}(H)$, which is the quotient of the tensor algebra $\bigoplus_{n\geq0}H^{\otimes_{\mathbf Z_p}n}$ by the relations given by the Lie bracket. The tensor algebra obtains an induced filtration by letting $h_1\otimes\cdots\otimes h_n$ be in $\Fil^s$ as long as $\omega(h_1)+\cdots+\omega(h_n)\geq s$. The inequality $\omega([g,h])\geq\omega(g)+\omega(h)$ then implies that the quotient $U_{\mathbf Z_p}(H)$ is still a filtered ring and actually a filtered $\mathbf Z_p$-algebra. One can check that $\gr U_{\mathbf Z_p}(H)\cong U_{\mathbf F_p[\varepsilon]}(\gr H)$, canonically.\\
	
	The logarithm $\log\colon H\to\Lie H$ extends to the universal enveloping algebras and by composing with (\ref{infaction}) and tensoring by $\mathcal O_L$ we obtain a map
	\begin{equation}\label{univHtodistr}U_{\mathcal O_L}(\mathcal O_L\otimes_{\mathbf Z_p}H)\to D(H,L), h\mapsto-\sum_{i\geq1}\frac{(1-[h])^i}i.\end{equation}
	Note that the source is a filtered $\mathcal O_L$-algebra by giving it the tensor filtration via its identification with $\mathcal O_L\otimes_{\mathbf Z_p}U_{\mathbf Z_p}(H)$.\\ The following proposition shows that the associated graded $\gr U_{\mathcal O_L}(\mathcal O_L\otimes_{\mathbf Z_p}H)$ is isomorphic to $\gr_{1/p}\mathcal O_L\llbracket H\rrbracket$, which is \cite[Thm.~(III, 2.3.3)]{Lazard1965}. Note that there are also similar (ungraded) versions of Lazard's isomorphism (\cite[\S2.2]{DPSinfchar}).
	\begin{proposition}\label{grisunivenv}
		The map (\ref{univHtodistr}) postcomposed with the map to $D_{1/p}(H,L)$ is an injective map of filtered $\mathcal O_L$-algebras. In fact, $\Fil_{1/p}^\bullet D_{1/p}(H,L)$ induces on $U_{\mathcal O_L}(\mathcal O_L\otimes_{\mathbf Z_p}H)$ its filtration defined above so that we have an injection \begin{equation*}
			\gr U_{\mathcal O_L}(\mathcal O_L\otimes_{\mathbf Z_p}H)\hookrightarrow\gr_{1/p}D_{1/p}(H,L).
		\end{equation*}
		The image of this map is $\gr_{1/p}\mathcal O_L\llbracket H\rrbracket$.
	\end{proposition}
	\begin{proof}
		One checks that the map is one of filtered $\mathcal O_L$-modules on the ordered basis $h_1,\dots,h_d$, so that the map is also of filtered $\mathcal O_L$-modules. To conclude all but the final statement, it is then sufficient to prove that the map of associated gradeds is injective. For this we use \begin{equation*}
			\gr U_{\mathcal O_L}(\mathcal O_L\otimes_{\mathbf Z_p}H)=U_{k_L[\varepsilon_L]}(k_L[\varepsilon_L]\otimes_{\mathbf F_p[\varepsilon]}\gr H)
		\end{equation*} and the fact that $k_L[\varepsilon_L]\otimes_{\mathbf F_p[\varepsilon]}\gr H$ is free over $k_L[\varepsilon_L]$ with basis given by the principal symbols of the $h_1,\dots,h_d$. Namely, each principal symbol $\sigma(h_i)$ is in degree $\omega(h_i)$ in $\gr U_{\mathcal O_L}(\mathcal O_L\otimes_{\mathbf Z_p}H)$, and the map sends $\sigma(h_i)$ to $\sigma_{1/p}(h_i-1)=\sigma_{1/p}(-\sum_{n\geq1}\frac{(1-[h_i])^n}n)$, which is in the same degree. Using Poincar\'e-Birkhoff-Witt $U_{k_L[\varepsilon_L]}(k_L[\varepsilon_L]\otimes_{\mathbf F_p[\varepsilon]}\gr H)$ can be made completely explicit as $k_L[\varepsilon_L]$-module: it is the symmetric algebra on $\sigma(h_1), \dots, \sigma(h_d)$. The map can then be described as \begin{equation*}
			\sum_{\alpha\in\mathbf Z^d}c_\alpha\sigma(h_1)^{\alpha_1}\cdots\sigma(h_d)^{\alpha_d}\mapsto\sum_{\alpha\in\mathbf Z^d}c_\alpha\sigma_{1/p}([h_1]-1)^{\alpha_1}\cdots\sigma_{1/p}([h_d]-1)^{\alpha_d}
		\end{equation*} and this is clearly injective by our description of $D_{1/p}(H,L)$ and its norm. At the same time we also find that the image is $\gr_{1/p}\mathcal O_L\llbracket H\rrbracket$.
	\end{proof}
	\begin{remark}
		Even when $H$ is just $p$-valued and not necessarily saturated the $\mathbf F_p[\varepsilon]$-module (where $\varepsilon$ corresponds to raising group elements to the $p$-th power) $\gr H\coloneqq\bigoplus_{s\in\mathbf R}\left\{h\in H\colon\omega(h)\geq s\right\}/\left\{h\in H\colon\omega(h)>s\right\}$ is still an $\mathbf F_p[\varepsilon]$-Lie algebra and one has $\gr_{1/p}\mathcal O_L\llbracket H\rrbracket\cong U_{k_L[\varepsilon_L]}(k_L[\varepsilon_L]\otimes_{\mathbf F_p[\varepsilon]}\gr H)$. See \cite[Thm.~(III, 2.3.3)]{Lazard1965}. It allows us to get the Noetherianity in the next corollary more generally.
	\end{remark}
	\begin{corollary}
		Let $H$ be any $p$-valued $p$-adic Lie group. Then the rings $\gr k_L\llbracket H\rrbracket$, $\gr_{1/p}\mathcal O_L\llbracket H\rrbracket$, $\gr_{1/p}L\llbracket H\rrbracket\cong\gr_{1/p}D_{1/p}(H,L)$ are all Noetherian.		
	\end{corollary}
	\begin{proof}
		Since $\gr_{1/p}L\llbracket H\rrbracket$ is a localization of $\gr_{1/p}\mathcal O_L\llbracket H\rrbracket$, of which $\gr k_L\llbracket H\rrbracket$ is a quotient, it suffices to show $\gr_{1/p}L\llbracket H\rrbracket$ is Noetherian. By the previous proposition (and remark) it is the universal enveloping algebra of a finite Lie algebra over $\gr\mathcal O_L$. We have the Poincar\'e-Birkhoff-Witt filtration \begin{equation*}
			\Fil^iU_{\gr\mathcal O_L}\left(\gr H\right)\coloneqq\Span_{\gr\mathcal O_L}\left\{x_1\cdot\cdots\cdot x_j\colon j\leq-i, x_k\in\gr H\right\}
		\end{equation*} for $i\in\mathbf Z$, with respect to which the associated graded ring is a polynomial algebra over $\gr\mathcal O_L$ in finitely many variables (of negative degree), and therefore Noetherian.\\ We now conclude by \cite[Prop.~I.7.1.2]{li1996zariskian}: a ring with a complete and separated filtration is (left, respectively right) Noetherian if its associated graded is and if the jumps of the filtration occur only in $\mathbf Z$ (this last part is implicitly taken into account via the definition of filtrations used by \cite{li1996zariskian}).
	\end{proof}
	\begin{remark}
		Assuming that $\omega$ takes values in $\mathbf Q$ one can deduce Noetherianity of $\mathcal O_L\llbracket H\rrbracket$, $D_{1/p}(H,L)$, etc.~using \cite[Prop.~I.7.1.2]{li1996zariskian} again, because in this case $\Fil^\bullet_{1/p}$ jumps only at $s\in\frac1M\mathbf Z$, for some $M\in\mathbf Z_{>0}$. Then one can also deduce flatness of $D_{1/p}(H,L)$ over $L\llbracket H\rrbracket$, in the same way as in the proof of \cite[Prop.~4.7]{cite-key}. Assuming $\omega$ to be ``strictly saturated'' as defined in the next section, one can then deduce, via (\ref{distralgdecomp}) below, that $D_{r_N}(H,L)$ is flat over $L\llbracket H\rrbracket$, for $r_N\coloneqq p^{-1/p^N}$.\\ Although this would not yield flatness of general $D_r(H,L)$ with $r\in p^{\mathbf Q}$, it would still extend \cite[Prop.~4.7]{cite-key} in the sense that their hypothesis ``(HYP)'' would not be necessary.\\
		In any case, we do not explore this further, because we in fact will not need flatness.
	\end{remark}
	\section{Groups of $p$-powers}\label{ppowersect}
	In this section, we assume the $p$-valuation on $H$ is saturated.\\ 
	
	Let $N\geq0$ be an integer and $r\in[1/p,1)$. The set of $p^N$-powers $$H^{p^N}\coloneqq\left\{h^{p^N}\colon h\in H\right\}$$ is an open subgroup of $H$ and it naturally has an induced \emph{saturated} $p$-valuation, namely $\omega-N$. Also notice that $h_1^{p^N},\dots,h_d^{p^N}$ gives an ordered basis. With respect to this, we can do the above constructions as well for $H^{p^N}$, to obtain filtered rings \begin{equation*}
		D_r(H^{p^N},L)=\left\{\sum_\alpha c_\alpha\mathbf b_{p^N}^\alpha\colon v_p(c_\alpha)-\frac{\log r}{\log p}\tau\alpha\to\infty\text{ as }\alpha\to\infty\right\},
	\end{equation*} where $\mathbf b_{p^N}^\alpha=([h_i^{p^N}]-1)^{\alpha_1}\cdot\cdots\cdot([h_d^{p^N}]-1)^{\alpha_d}$. We also get the induced filtrations $\Fil_r^\bullet\mathcal O_L\llbracket H^{p^N}\rrbracket$ and $\Fil_r^\bullet L\llbracket H^{p^N}\rrbracket$.\\  
	We have a decomposition\begin{equation*}
		\mathcal O_L\llbracket H\rrbracket=\bigoplus_{H^{p^N}h\in H^{p^N}\backslash H}\mathcal O_L\llbracket H^{p^N}\rrbracket\cdot[h]
	\end{equation*} and likewise for $L\llbracket H\rrbracket$ and $D(H,L)$. However, the comparison is a bit more subtle for the distribution algebras of fixed radius (and therefore also the filtrations induced from them). To facilitate their comparison we prefer to consider only the following radii of analyticity:\begin{equation*}
		r_n\coloneqq p^{-1/p^n}, n\in\mathbf Z_{\geq0}.
	\end{equation*}
	We also need the following notion, which is slightly stronger than saturatedness (a $p$-valuation on a $p$-adic Lie group is saturated if every element of an ordered basis has $p$-valuation $\leq p/(p-1)$ by \cite[Prop.~26.11]{Schneider2011}).
	\begin{definition}
		Let $H'$ be a $p$-adic Lie group with $p$-valuation $\omega'$ and ordered basis $h'_1,\dots,h'_n$. We call $\omega'$ \emph{strictly saturated} if $\omega(h'_i)<p/(p-1)$ for all $i$.
	\end{definition}
	It is easy to see that the induced $p$-valuations on the groups of $p^N$-powers are also strictly saturated if the original one is. Under the assumption of strictly saturatedness we have the following relation:
			\begin{lemma} Suppose $\omega$ is strictly saturated. Then
				\begin{equation*}L\llbracket H^{p^N}\rrbracket\cap\Fil_{r_N}^\bullet L\llbracket H\rrbracket=\Fil_{r_0}^\bullet L\llbracket H^{p^N}\rrbracket.\end{equation*}
				Consequently, $\gr_{1/p}L\llbracket H^{p^N}\rrbracket$ and $\gr_{1/p}\mathcal O_L\llbracket H^{p^N}\rrbracket$ identify as subrings of, respectively, $\gr_{r_N}L\llbracket H\rrbracket$ and $\gr_{r_N}\mathcal O_L\llbracket H\rrbracket$. Under this identification, \begin{equation*}
					\gr_{r_N}\mathcal O_L\llbracket H\rrbracket\ni\sigma_{r_N}\left(\left[h_i\right]-1\right)^{p^N}=\sigma_{1/p}\left(\left[h_i^{p^N}\right]-1\right)\in\gr_{1/p}\mathcal O_L\llbracket H^{p^N}\rrbracket.
				\end{equation*}
			\end{lemma}	
			\begin{proof}
				Using the identity $a^p-1=\sum_{j=1}^p\binom pj(a-1)^j$ repeatedly one can show by induction that modulo the ideal of $\mathcal O_L\llbracket H\rrbracket$ generated by the $p^k([h_i]-1)^{p^{N-k}}$ for $k\in\{1,\dots, N\}$ one has the congruence \begin{equation*}
					[h_i^{p^N}]-1\equiv([h_i]-1)^{p^N}
				\end{equation*} for each $i$. The $r_N$-valuation of a $p^k([h_i]-1)^{p^{N-k}}$ equals $k+\omega(h_i)/p^k$. Since $\omega$ is strictly saturated, it follows that the $r_N$-valuation of $[h_i^{p^N}]-1$ is the same as that of $([h_i]-1)^{p^N}$, namely $\omega(h_i)$. This is $\omega(h_i^{p^N})-N$ and therefore also the $r_0$-valuation of $[h_i^{p^N}]-1$ in $L\llbracket H^{p^N}\rrbracket$. Moreover, we even see that the principal symbols of $[h_i^{p^N}]-1$ and $([h_i]-1)^{p^N}$ in $\gr_{r_N}L\llbracket H\rrbracket$ are the same. This will already give us the last part of the lemma.\\
				Using this observation, one checks that a general element $\sum_\alpha c_\alpha\mathbf b_{p^N}^\alpha$ of $L\llbracket H^{p^N}\rrbracket$ has the same principal symbol in $\gr_{r_N}L\llbracket H\rrbracket$ as $\sum_\alpha c_\alpha\mathbf b^{p^N\alpha}$. Therefore, $\sum_\alpha c_\alpha\mathbf b_{p^N}^\alpha$ has $r_N$-valuation $\inf_\alpha v_p(c_\alpha)+\frac1{p^N}\cdot\tau(p^N\alpha)=\inf_\alpha v_p(c_\alpha)+\tau\alpha$ in $L\llbracket H\rrbracket$, which is also the $r_0$-valuation it has in $L\llbracket H^{p^N}\rrbracket$. We conclude.
			\end{proof}
			It follows that the closure of $L\llbracket H^{p^N}\rrbracket$ in $D_{r_N}(H,L)$ equals $D_{r_0}(H^{p^N},L)$. So $D_{r_N}(H,L)=\bigoplus_{H^{p^N}h\in H^{p^N}\backslash H}D_{r_0}(H^{p^N},L)\cdot[h]$.\\
			As a consequence, we find by induction that in general $D_{r_n}(H^{p^N},L)\subset D_{r_{n+N}}(H,L)$ and \begin{equation}\label{distralgdecomp}
				D_{r_{n+N}}(H,L)=\bigoplus_{H^{p^N}h\in H^{p^N}\backslash H}D_{r_n}(H^{p^N},L)\cdot[h]
			\end{equation} in the strictly saturated case.\\
			We conclude with two more useful observations. The first generalizes an observation made in the previous proof of the previous lemma.
			\begin{lemma}\label{compatibilityppowers}Suppose $\omega$ is strictly saturated. Let $\sum_ic_i\left(\left[h_i\right]-1\right)\in\mathcal O_L\llbracket H\rrbracket$ be an element for which every non-zero $c_i$ is in $\mathcal O_L^\times$ and for which for all $N$ the $\sigma_{r_N}\left(\left[h_i\right]-1\right)$ with $c_i\neq0$ all commute and have the same degree. Then we have for all $N,N'\in\mathbf Z_{\geq0}$ that the element \begin{equation*}
					\sigma_{r_{N+N'}}\left(\sum_ic_i\left(\left[h_i\right]-1\right)\right)^{p^N}\in\gr_{r_{N+N'}}\mathcal O_L\llbracket H\rrbracket
					\end{equation*} equals
				\begin{equation*}\sigma_{r_{N'}}\left(\sum_ic_i^{p^N}\left(\left[h_i^{p^N}\right]-1\right)\right)\in\gr_{r_{N'}}\mathcal O_L\llbracket H^{p^N}\rrbracket\subset\gr_{r_{N+N'}}\mathcal O_L\llbracket H\rrbracket.
				\end{equation*}
			\end{lemma}
			\begin{proof}
				By the assumptions, we already find that 
				\begin{equation*}
					\sigma_{r_{N+N'}}\left(\sum_ic_i\left(\left[h_i\right]-1\right)\right)^{p^N}=\sum_ic_i^{p^N}\sigma_{r_{N+N'}}\left(\left[h_i\right]-1\right)^{p^N}.
				\end{equation*}
				(Note that we can also take out the sum on the right hand side of the equality we want to prove.) But then we conclude by \begin{equation*}
					\sigma_{r_{N+N'}}\left(\left[h_i\right]-1\right)^{p^N}=\sigma_{r_{N'}}\left(\left[h_i^{p^N}\right]-1\right),
				\end{equation*}which is a consequence of the preceding lemma.
			\end{proof}
			Finally, by passing to $H^p$ one can put oneself in a situation in which the above commutativity condition is automatic:
			\begin{lemma}\label{commutativity}
				Suppose $\omega$ is saturated. For $N\geq1$ and $r\in[1/p,1)$ the rings $\gr_r\mathcal O_L\llbracket H^{p^N}\rrbracket$ are commutative polynomial $k_L$-algebras in the variables $\sigma_r([h_i^{p^N}]-1)$ and $\varepsilon_L$.
			\end{lemma}
			\begin{proof}
				The case $N=1$ suffices. The group $H^p$ is considered with ordered basis $h_1^p,\dots,h_d^p$ and $p$-valuation $\omega-1$. We find that \begin{equation*}
					\omega\left([h_i^p,h_j^p]\right)-1\geq\omega(h_i^p)+\omega(h_j^p)-1>\begin{cases}
						\frac{\omega(h_i^p)-1+\omega(h_j^p)-1}p+1&*,\\
						\omega(h_i^p)-1+\omega(h_j^p)-1&**.
					\end{cases}
				\end{equation*} 
				For $r>1/p$ the proof of \cite[Lemma 4.4]{cite-key} works for the commutativity: this lemma still holds if one replaces their ``(HYP)'' by the hypothesis on their $(G,\omega)$ and its ordered basis $h_1,\dots,h_d$ that it is saturated and $\omega([h_i,h_j])>\frac{\omega(h_i)+\omega(h_j)}p+1$; the bound * shows that we are indeed in this situation.\\
				For $r=1/p$ we instead use Proposition \ref{grisunivenv} that $\gr_{1/p}\mathcal O_L\llbracket H^p\rrbracket\cong U_{\mathcal O_L}\left(\mathcal O_L\otimes_{\mathbf Z_p}H\right)\cong U_{k_L[\varepsilon_L]}\left(k_L[\varepsilon_L]\otimes_{\mathbf F_p[\varepsilon]}\gr H\right)$. The commutativity claim then follows by $\gr H$ being a commutative Lie algebra, which is a direct consequence of the bound **.\\
				In all cases one concludes that these are the described polynomial rings in the same way as in \cite[Thm.~4.5]{cite-key}.
			\end{proof}
			\begin{remark}\label{notuniform}
				Often in the literature people prefer to work with the more restrictive notion of ``uniform'' $p$-adic Lie groups $H$ instead of $p$-valued $H$. These are $H$ for which $H/H^p$ is Abelian, and one can put a $p$-valuation on them that sends each ordered basis element to $1$, see \cite[Thm.~4.5]{ddms} and also the remark before \cite[Lemma 4.4]{cite-key}.\\ 
				When $H$ is saturated, the bound * in the proof of the lemma above shows that $H^p$ is actually uniform. Still, the $p$-valuation $\omega-1$ on $H^p$ rather than one sending all ordered basis elements to $1$ is potentially much more interesting to us, even if it might be a bit harder to work with. Indeed, our main theorem will depend crucially on the fact (Lemma \ref{lieelts}) that Casimir elements of $\gl_2 K$ are in some sense homogenous with respect to a certain (``non-uniform'') $p$-valuation.
			\end{remark}
			\section{$\varepsilon$-torsion freeness}\label{torfreesect}
			In this section, $H$ still has $p$-valuation $\omega$, but it is not required to be saturated anymore.\\
			
			Let $M_\circ$ be a finite type $\mathcal O_L\llbracket H\rrbracket$-module without $p$-torsion. Fix a surjection $\mathcal O_L\llbracket H\rrbracket^{\oplus N_0}\twoheadrightarrow M_\circ$ with kernel $N_\circ$. In this section we discuss the following result, essentially due to \cite{cite-key}:
			\begin{proposition}\label{epstorsfree}
				For $N\gg0$ the associated graded module $\gr_{r_N}M_\circ$ is $\varepsilon$-torsion free.
			\end{proposition}
			In \cite{cite-key} this is only proven in the case $M_\circ$ has one generator and under their assumption ``(HYP)'', as part of their proof of \cite[Thm.~4.11]{cite-key}. However, their proof almost goes through verbatim. Nevertheless, we have chosen to give it in full here, because of the crucial role it plays in our later arguments. 
			Write $\overline N=N_\circ/\left(N_\circ\cap\varpi_L\mathcal O_L\llbracket H\rrbracket^{\oplus N_0}\right)$ and, given an $N\geq0$, give it the induced filtration from $\Fil_{r_N}^\bullet k_L\llbracket H\rrbracket^{\oplus N_0}$ and write $\gr_{r_N}\overline N$ for its associated graded. We warn the reader that this is the one and only time we consider a filtration (on $\overline N$) that is not induced by a surjection from some finite free $k_L\llbracket H\rrbracket$-module.\\
			As do \cite{cite-key} we will need the following two lemma's:
			\begin{lemma}\label{vi}
				There exist $v_1,\dots, v_\ell\in N_\circ\setminus\varpi_L\mathcal O_L\llbracket H\rrbracket^{\oplus N_0}$ and, for each $N\geq0$, constants $s_{i,N}\in\mathbf R$ such that\begin{itemize}
					\item for each $N\geq0$, $\overline v_i\in \Fil^{s_{i,N}}_{r_N}(k_L\llbracket H\rrbracket^{\oplus N_0})\setminus \Fil^{>s_{i,N}}_{r_N}(k_L\llbracket H\rrbracket^{\oplus N_0})$
					\item for each $N\geq0$ and $s\in\mathbf R$, $\Fil^s_{r_N}\overline N=\sum_{i=1}^\ell\Fil_{r_N}^{s-s_{i,N}}(k_L\llbracket H\rrbracket)\cdot\overline v_i$. 
				\end{itemize}
			\end{lemma}
			\begin{proof}
				Recall first that $\gr_{r_N}k_L\llbracket H\rrbracket^{\oplus N_0}$ and therefore $\gr_{r_N}\overline N$  do not depend on $N$ except for the scaling of the grading. So the result is independent of $N$ and the constants $s_{i,N}$ depend only on $N$ because they need to be scaled appropriately.\\
				The existence of the $\overline v_i$ (and therefore the $v_i$ which we choose to be arbitrary lifts) is then an immediate consequence of the Noetherianity of $\gr k_L\llbracket H\rrbracket$: this implies $\gr\overline N$ is finitely generated over $\gr k_L\llbracket H\rrbracket$ and we may take these generators to be principal symbols, which we can then lift back to get the $\overline v_i$. A priori we then just get, for all $s\in\mathbf R$, that \begin{equation*}
					\Fil^s_{r_N}\overline N=\sum_{i=1}^\ell\Fil_{r_N}^{s-s_{i,N}}(k_L\llbracket H\rrbracket)\cdot\overline v_i+\Fil_{r_N}^{>s}\overline N,
				\end{equation*} but since $\overline N$ is separated and complete for its filtration, the second bullet point indeed holds by an approximation argument, using that the set of jumps of the filtration $\Fil_{r_N}^\bullet\overline N$ is contained in the discrete subset $\sum_i\omega(h_i)\mathbf Z_{\geq0}$ of $\mathbf R$.
			\end{proof}
			The next ingredient is the following lemma, generalizing \cite[Lemma 4.12]{cite-key}:
			\begin{lemma}\label{rvi}
				Let $v\in\mathcal O_L\llbracket H\rrbracket^{\oplus N_0}\setminus\varpi_L\mathcal O_L\llbracket H\rrbracket^{\oplus N_0}$. There is an $r_v\in[1/p,1)$ such that for all $r_v\leq r<1$ and all $s\in\mathbf R$: $v\in\Fil_r^s\left(\mathcal O_L\llbracket H\rrbracket^{\oplus N_0}\right)+\varpi_L\mathcal O_L\llbracket H\rrbracket^{\oplus N_0}$ implies $v\in\Fil_r^s\mathcal O_L\llbracket H\rrbracket^{\oplus N_0}$.
			\end{lemma}
			\begin{proof}
				Write $v=\left(\sum_\alpha c_{i,\alpha}\mathbf b^\alpha\right)_{i=1}^{N_0}$. The set $\{(\alpha,i)\colon c_{i,\alpha}\in\mathcal O_L^\times\}$ is non-empty, take in it an $(\alpha,i)$ for which $\tau\alpha$ is minimal. Now choose $r_v$ such that $-\frac{\log r_v}{\log p}\cdot\tau\alpha\in(0,1/e_L)$.
			\end{proof}
			Now we are ready for the main proof.
			\begin{proof}[Proof of Proposition \ref{epstorsfree}]
				Let $N\geq0$ such that $r_N\geq\max_ir_{v_i}$ where the $r_{v_i}$ are constants chosen from Lemma \ref{rvi} for the $v_i$ from Lemma \ref{vi}. This ensures that $v_i\in\Fil^{s_{i,N}}_{r_N}\mathcal O_L\llbracket H\rrbracket^{\oplus N_0}$ for all $i$. We will prove that $\gr_{r_N}M_\circ$ has no $\varepsilon$-torsion by proving it has no $\varepsilon_L$-torsion.\\
				For this, suppose $x\in\Fil_{r_N}^sM_\circ$ and $\varpi_Lx\in\Fil_{r_N}^{>s+1/e_L}M_\circ$. It suffices to show $x\in\Fil_{r_N}^{>s}M_\circ$. Write \begin{equation*}
					\varpi_Lx=y+u, \text{ where }y\in\Fil_{r_N}
					^{>s+1/e_L}\left(\mathcal O_L\llbracket H\rrbracket^{\oplus N_0}\right), u\in N_\circ.	\end{equation*}
				Our goal is to find $y'\in\mathcal O_L\llbracket H\rrbracket^{\oplus N_0}$ and $u'\in\Fil_{r_N}^{>s+1/e_L}N_\circ$ such that $y=\varpi_Ly'+u'$. Namely, once this is done, we will have that $\varpi_Ly'=y-u'\in\Fil_{r_N}^{>s+1/e_L}\mathcal O_L\llbracket H\rrbracket^{\oplus N_0}$ so that $y'\in\Fil_{r_N}^{>s}\mathcal O_L\llbracket H\rrbracket^{\oplus N_0}$ and $\varpi_L(x-y')=u+u'\in N_\circ$. The $\varpi_L$-torsion freeness of $M_\circ=\mathcal O_L\llbracket H\rrbracket^{\oplus N_0}/N_\circ$ then implies $x\equiv y'\mod N_\circ$, so that $x\in\Fil_{r_N}^{>s}M_\circ$ as we want to show.\\
				Because $\overline y=-\overline u\in\Fil_{r_N}^{>s+1/e_L}\overline N$, we can find by Lemma \ref{vi}, for each $i$, $c_i\in\Fil_{r_N}^{>s+1/e_L-s_{i,N}}\mathcal O_L\llbracket H\rrbracket$ such that \begin{equation*}
					y\equiv c_1v_1+\dots+c_\ell v\ell\mod\varpi_L\mathcal O_L\llbracket H\rrbracket^{\oplus N_0}.
				\end{equation*}
				That is, $y=\varpi_Ly'+u'$ for some $y'\in\mathcal O_L\llbracket H\rrbracket^{\oplus N_0}$ and where $u'\coloneqq c_1v_1+\dots+c_\ell v\ell\in\Fil_{r_N}^{>s+1/e_L}N_\circ$.
			\end{proof}
			\begin{corollary}
				For $N\gg0$ the natural map $\gr_{r_N}M_\circ\to\gr_{r_N}M=\gr_{r_N}M_\circ[1/\varepsilon]$ is an injection.
			\end{corollary}
			\begin{corollary}\label{injectionsforNsuffbig}
				For $N\gg0$, the natural maps $\gr_{r_N}M_\circ\to\gr_{r_N}M_{r_N}$ and $M\to M_{r_N}$ are injections.
			\end{corollary}
			\begin{proof}
				Take $N\gg0$ so that $\gr_{r_N}M_\circ$ is $\varepsilon$-torsion free. By i) of Lemma \ref{commutation}, $\gr_{r_N}M=\gr_{r_N}M_{r_N}$. By ii) of the same lemma $\gr_{r_N}M_\circ$ injects into $\gr_{r_N}M_\circ$, because of $\varepsilon$-torsion freeness. This implies that $M_\circ$ itself injects into $M_{r_N}$: suppose $x\in M_\circ$ is mapped to $0$, then $x\in\cap_{s\geq0}\Fil_{r_N}^sM_\circ$, but this intersection is $0$, which follows from continuity of $\mathcal O_L\llbracket H\rrbracket^{\oplus N_0}\twoheadrightarrow M_\circ$ for example.
			\end{proof}
			\begin{corollary}
				Let $G$ be a $p$-adic Lie group, containing $H$ as an open subgroup. Let $\Pi$ be an admissible $L$-Banach representation of $G$. Then for $N\gg0$ the $r_N$-analytic vectors (taken with respect to $H$) $\Pi^\text{$r_N$-la}\coloneqq\left(\left(\Pi^*\right)_{r_N}\right)^*$ are dense in $\Pi$.
			\end{corollary}
			\begin{proof}
				Take $N\gg0$ so that $\Pi^*\to\left(\Pi^*\right)_{r_N}=\left(\Pi^\text{$r_N$-la}\right)^*$ is injective. Suppose, for a contradiction, that $\Pi^\text{$r_N$-la}\subset\Pi$ is not dense. Let $V$ be the closure of $\Pi^\text{$r_N$-la}$ inside $\Pi$ and $w\in\Pi\setminus V$. Denote by $\|-\|$ the norm on $\Pi$. One can define a continuous $L$-valued linear form $\ell$ on $Lw+V\cong Lw\oplus V$ with kernel $V$ and $|\ell(w)|\leq\inf\{\|w+v\|\colon v\in V\}>0$.\\
				Then this form extends to an element $\ell'\in\Pi^*$ by Hahn-Banach (\cite[Prop.~9.2]{Schneider2002}). The form $\ell'$ is non-zero but in the kernel of $\Pi^*\to\left(\Pi^*\right)_{r_N}$ by construction, so that we arrive at our desired contradiction.
			\end{proof}			
			\section{Bounding the Gelfand-Kirillov dimension}\label{sectionAppli}
			Now we will turn to an application of the above generalities: let $p>2$ and $K$ be a $p$-adic field, then we will show that any admissible $p$-adic Banach representation with infinitesimal character of $\GL_2K$ or the units of the quaternion algebra over $K$ has Gelfand-Kirillov dimension at most $[K\colon\mathbf Q_p]$. Before we explain our setup, we briefly remind the reader on the notion of Gelfand-Kirillov dimension.
			\subsection{Reminder on the Gelfand-Kirillov dimension}\label{GKsection}
			Let $G$ be a $p$-adic Lie group and $H$ any compact open subgroup. Let $C$ be one of the rings of coefficients $k_L,\mathcal O_L,$ or $L$. Let $M$ be a finitely generated $C\llbracket H\rrbracket$-module. The following notion essentially captures the Gelfand-Kirillov dimension: the \emph{grade} $j_{C\llbracket H\rrbracket}(M)$ of $M$ is defined to be \begin{equation*}
				j_{C\llbracket H\rrbracket}(M)\coloneqq\inf\left\{i\colon\Ext^i_{C\llbracket H\rrbracket}(M,C\llbracket H\rrbracket)\neq0\right\}.
			\end{equation*}
			The definition of the Gelfand-Kirillov dimension is as follows: the \emph{Gelfand-Kirillov dimension} of $M$ is defined to be \begin{equation*}
				d_{C\llbracket H\rrbracket}(M)\coloneqq\begin{cases}
					\dim H-j_{C\llbracket H\rrbracket}(M)&\text{if $C$ is $k_L$ or $L$,}\\
					\dim H+1-j_{C\llbracket H\rrbracket}(M)&\text{if $C=\mathcal O_L$.}
				\end{cases}
			\end{equation*}
			If $\Pi$ (resp.~$\pi$) is an admissible $L$-Banach (resp.~ admissible smooth $k_L$-)representation of $G$, then its continuous dual is a finitely generated $L\llbracket H\rrbracket$-(resp.~$k_L\llbracket H\rrbracket$-)module and we define its Gelfand-Kirillov dimension to be that of its dual.
			\begin{lemma}\label{gklemma}
				Let $M$ be a finitely generated $L\llbracket H\rrbracket$-module, $L'/L$ a finite field extension, and $H'\subset H$ an open subgroup. Then \begin{equation*}
					d_{L\llbracket H\rrbracket}(M)=d_{L'\llbracket H\rrbracket}(L'\otimes_LM)=d_{\mathbf Q_p\llbracket H\rrbracket}(M)=d_{\mathbf Q_p\llbracket H'\rrbracket}(M),
				\end{equation*} and the analogous statements hold for $M_\circ$ and $\overline M$ (over $\mathcal O_L$ and $k_L$ respectively).
			\end{lemma}
			\begin{proof}
				This follows directly from $L'\llbracket H\rrbracket$ being finite free over $L\llbracket H'\rrbracket$ and the fact that $\Ext^i_{L\llbracket H\rrbracket}(M,L\llbracket H\rrbracket)$ is computed by applying $\Hom_{L\llbracket H\rrbracket}(-,L\llbracket H\rrbracket)$ to a resolution of $M$ by finite free $L\llbracket H\rrbracket$-modules. The reasoning for $M_\circ$ and $\overline M$ is analogous. 
			\end{proof}
			Consequently, the Gelfand-Kirillov dimension does not depend on the choice of compact open subgroup $H$ or of coefficients. This means that one can try to choose $H$ in a way that facilitates the calculation of the Gelfand-Kirillov dimension. For example, Lemma \ref{commutativity} puts us in the situation where the following can be used:
			\begin{proposition}[{\cite[Prop.~3.5]{dospinescu2023gelfandkirillovdimensionpadicjacquetlanglands}}]\label{gradeq}
				Suppose that $k_L\llbracket H\rrbracket$ is complete for some filtration with respect to which $\gr(k_L\llbracket H\rrbracket)$ is a commutative Noetherian regular ring. \\
				Then for any $p$-torsion free $\mathcal O_L\llbracket H\rrbracket$-module $M_\circ$ with $M=M_\circ[1/p]$ and reduction modulo $\varpi_L$ equal to $\overline M$: 
				\begin{equation*}
					j_{\mathcal O_L\llbracket H\rrbracket}(M_\circ)=j_{L\llbracket H\rrbracket}(M)=j_{k_L\llbracket H\rrbracket}(\overline M).
				\end{equation*}
				Moreover, if $\overline M$ has a \emph{good filtration}, that is a filtration compatible with the one on $k_L\llbracket H\rrbracket$ for which the associated graded $\gr\overline M$ is finitely generated over $\gr k_L\llbracket H\rrbracket$, then the above quantities are all equal to \begin{equation*}
					j_{\gr k_L\llbracket H\rrbracket}(\gr\overline M)\coloneqq\inf\left\{i\colon\Ext^i_{\gr k_L\llbracket H\rrbracket}\left(\gr\overline M,\gr k_L\llbracket H\rrbracket\right)\neq0\right\}.
				\end{equation*}
			\end{proposition}
			So let us indeed put ourselves in the situation of Lemma \ref{commutativity}, so that $H$ has a strictly saturated $p$-valuation $\omega$ and ordered basis $h_1,\dots,h_d$ with respect to which we define the $r$-analytic filtrations, for $r\in[1/p,1)$. Then the lemma tells us that $\gr_r\mathcal O_L\llbracket H^{p^N}\rrbracket$, for $N>0$, are polynomial rings over $k_L$ in $\sigma_r([h_i^{p^N}]-1)$ and $\varepsilon_L$, so that $\gr k_L\llbracket H^{p^N}\rrbracket$ is a polynomial $k_L$-algebra in $d=\dim H$ variables and in particular Noetherian regular, putting us in the situation of the above proposition. We see that computing the Gelfand-Kirilov dimension therefore boils down to computing the grade $j_{\gr(k_L\llbracket H^{p^N}\rrbracket)}(\gr(\overline M))$.\footnote{The filtration on $\overline M$ coming from a surjection from $k_L\llbracket H\rrbracket^{\oplus N_0}$, as we are used to, is indeed a good filtration over $k_L\llbracket H^{p^N}\rrbracket$: this follows from $\gr k_L\llbracket H\rrbracket$ being a finite $\gr k_L\llbracket H^{p^N}\rrbracket$-module, which we leave to the reader to check.}\\
			To this end, we can apply \cite[III\S4.1, Thm.~7 ``Generalized Roos Theorem'']{li1996zariskian}, which gives: for any finitely generated $\gr k_L\llbracket H^{p^N}\rrbracket$-module $A$ \begin{equation}\label{Krull}
				j_{\gr k_L\llbracket H^{p^N}\rrbracket}(A)=\dim H-\Kdim A,
			\end{equation}
			where $\Kdim A$ is the Krull dimension of $A$, which is the dimension of its support in $\Spec \gr k_L\llbracket H^{p^N}\rrbracket$, that is, the Krull dimension of the ring $(\gr k_L\llbracket H^{p^N}\rrbracket)/\Ann A$.\\
			In conclusion, we find that the Gelfand-Kirillov dimension of a finitely generated $k_L\llbracket H\rrbracket$-module $\overline M$ equals the Krull dimension of $\gr\overline M$ as a $\gr k_L\llbracket H^{p^N}\rrbracket$-module, for any $N\geq1$.
			\subsection{Setup}\label{setupappl}
			Let $p>2$ and let $K$ be a finite extension of $\mathbf Q_p$, with uniformizer $\varpi_K$. We will apply our preparations of the previous sections to two $p$-adic Lie groups $G$ with a saturated $p$-valued open subgroup $H$:
			\begin{equation*}
				G=\begin{cases}
					\GL_2K&\textbf{in the $\GL_2$-case,}\\
					D^\times&\textbf{in the quaternion case.}
				\end{cases}
			\end{equation*} 
			Here $D$ is the quaternion algebra over $K$. It is generated as $K$-algebra by elements $\Pi$ and $\sqrt a$, satisfying:\begin{itemize}
				\item $\Pi^2=\varpi_K$,
				\item $a\in\mathcal O_K^\times$ is an element whose square root generates the degree $2$ unramified extension of $K$, and we take it such that $\mathbf Q_p[\sqrt a]$ is unramified over $\mathbf Q_p$,
				\item $\Pi$ and $\sqrt a$ anti-commute.
			\end{itemize}
			We write $\mathcal O_D=\mathcal O_K[\Pi,\sqrt a]$, the maximal order over $\mathcal O_K$ in $D$. It consists of all elements of non-negative valuation in $D$.\\
			We first define the following subgroup 
			\begin{equation*}
				H^{1/p}=\begin{cases}
					\begin{pmatrix}
						1+p\varpi_K\mathcal O_K&p\mathcal O_K\\p\varpi_K\mathcal O_K&1+p\varpi_K\mathcal O_K
					\end{pmatrix}&\text{$\GL_2$-case,}\\
					1+p\Pi\mathcal O_D&\text{quaternion case.}
				\end{cases}
			\end{equation*}
			In each case we can equip $H^{1/p}$ with a strictly saturated $p$-valuation $\omega$. The quaternion case is the most straightforward: \begin{equation}\label{quatpval}
				\omega(x)\coloneqq v_p(x-1)-C_\text{quat}, \text{ for }x\in1+p^2\Pi\mathcal O_D.
			\end{equation}
			Here one has to substract a constant $C_\text{quat}$ in order to get a $p$-valuation that is strictly saturated: for this the image of an ordered basis has to be in the interval $\left(1/(p-1),p/(p-1)\right)$, so that $C_\text{quat}$ has to be in the interval \begin{equation*}\left(1-\frac1{p-1}-\frac1{2e_K},1-\frac1{p-1}+\frac1{2e_K}\right).\end{equation*} While the precise value is not important at all and not even used, we still set \begin{equation*}
				C_\text{quat}\coloneqq1-1/(p-1).
			\end{equation*}
			Using the logarithm $\log\colon1+p\mathcal O_D\to p^2\mathcal O_D$, which is a homeomorphism transforming $p$-powers into multiples of $p$, one can then quickly check that (\ref{quatpval}) indeed yields a $p$-valuation. \\
			In the $\GL_2$-case one first observes that \begin{equation*}
				\tilde\omega\begin{pmatrix}
					a&b\\c&d
				\end{pmatrix}\coloneqq\min\{v_p(a), v_p(b), v_p(c), v_p(d)\}
			\end{equation*} is a $p$-valuation on $\{g\in\GL_2K[\sqrt{\varpi_K}]\colon\tilde\omega(g)>1/(p-1)\}$, again by considering the logarithm which is a homeomorphism from this set onto its image. The map $x\mapsto\begin{pmatrix}
				1&\\&1/\sqrt{\varpi_K}
			\end{pmatrix}x\begin{pmatrix}
				1&\\&\sqrt{\varpi_K}
			\end{pmatrix}$ embeds $H^{1/p}$ into $\{g\in\GL_2K[\sqrt{\varpi_K}]\colon\tilde\omega(g)>1/(p-1)\}$ and this way we finally find the following strictly saturated $p$-valuation on $H^{1/p}$:
			\begin{equation}\label{gl2pval}
				\omega\begin{pmatrix}
					a&b\\c&d
				\end{pmatrix}\coloneqq\min\left\{v_p(a-1), v_p(b)+\frac1{2e_K}, v_p(c)-\frac1{2e_K}, v_p(d-1)\right\}-C_{\GL_2}
			\end{equation}
			The substraction is again to ensure strict saturatedness and this time the constant \begin{equation*}
				C_{\GL_2}\coloneqq1-\frac1{p-1}+\frac1{4e_K}
			\end{equation*} could have been any in the interval \begin{equation*}
				\left(1-\frac1{p-1},1-\frac1{p-1}+\frac1{2e_K}\right).
			\end{equation*}\\
			
			As our notation suggests, we will work with the group $H\coloneqq\left(H^{1/p}\right)^p$ and the induced $p$-valuation $\omega-1$ on it. The reason for introducing $H$ as we did, is that now the commutativity of all graded algebras associated to the various completed group rings of $H$ is automatic, by Lemma \ref{commutativity}. Thus checking the conditions of Lemma \ref{compatibilityppowers} becomes easier, and one does not need to take a group of $p$-powers anymore to compute the Gelfand-Kirillov dimension as a Krull dimension as discussed in the previous section.\\
			
			\textbf{Assumption on the field of coefficients $L$:} we assume our $p$-adic field of coefficients $L$ is big enough in the following sense. In the $\GL_2$-case it contains the normal closure of $K$ over $\mathbf Q_p$, denoted by $K'$, and in the quaternion case it should contain the normal closure $K'[\sqrt a]$ of $K[\sqrt a]$ over $\mathbf Q_p$. Here $\sqrt a$ is the same element as in the description of the generators of the quaternions $D$ over $K$. We will immediately use this assumption for the decomposition of the next section.\\
			Since the Gelfand-Kirillov dimension of a $p$-adic Banach representation does not change under extension of scalars by Lemma \ref{gklemma}, this assumption will not affect the generality of our bound.
			\subsection{Decomposing $L\otimes_{\mathbf Q_p}\Lie G$}
			The Lie algebra \begin{equation*}
				\gl_2L\cong L\otimes_K\Lie G
			\end{equation*}
			is generated over $L$ by the elements \begin{equation*}
				e=\begin{pmatrix}
					0&1\\0&0
				\end{pmatrix},f=\begin{pmatrix}
					0&0\\1&0
				\end{pmatrix},h=\begin{pmatrix}
					1&0\\0&-1
				\end{pmatrix},z=\begin{pmatrix}
					1&0\\0&1
				\end{pmatrix}.
			\end{equation*}
			In the quaternion case, when $\Lie G=D$, we use that \begin{gather}\label{quatliealg}
				\begin{aligned}
					M_2 K[\sqrt a]&\to K[\sqrt a]\otimes_KD,\\
					e&\mapsto\frac1{2\varpi_K}\left(1\otimes\Pi+\frac1{\sqrt a}\otimes\sqrt a\Pi\right),\\
					f&\mapsto\frac12\left(1\otimes\Pi-\frac1{\sqrt a}\otimes\sqrt a\Pi\right),\\
					h&\mapsto\frac1{\sqrt a}\otimes\sqrt a,\\
					z&\mapsto1
				\end{aligned}
			\end{gather} is an isomorphism of $K[\sqrt a]$-algebras.\\
			In order to use this description for $L\otimes_{\mathbf Q_p}\Lie G$ we first decompose $L\otimes_{\mathbf Q_p}K$ in the standard way: \begin{equation*}
				L\otimes_{\mathbf Q_p}K\cong\bigoplus_{\rho\colon K\hookrightarrow L}L\otimes_{\rho,K}K.
			\end{equation*} Where the sum ranges over the embeddings $\rho\colon K\hookrightarrow L$ and where in $L\otimes_{\rho,K}K$ we consider the left factor as $K$-module via $\rho$. We denote by $1_\rho$ the idempotent corresponding to $\rho$, that is, the idempotent corresponding via the isomorphism above to the element $(\ell_{\rho'})_{\rho'}$ of $\bigoplus_{\rho'}L\otimes_{\rho',K}K$ with $\ell_\rho=1$ and $\ell_{\rho'}=0$ if $\rho'\neq\rho$.\\
			This way we get a decomposition \begin{equation}\label{Liedecomp}L\otimes_{\mathbf Q_p}\Lie G\cong\bigoplus_{\rho\colon K\hookrightarrow L}L\otimes_{\rho,K}\Lie G\cong\bigoplus_{\rho\colon K\hookrightarrow L}L\otimes_{\rho,K}\gl_2K.\end{equation}
			In the quaternion case, we will always use the isomorphism $L\otimes_{\rho,K}\Lie G\cong L\otimes_{\rho,K}\gl_2K$ obtained by replacing each ``$\sqrt a$'' by ``$\Frob^k\sqrt a$'' in the first tensor factor of (\ref{quatliealg}), where $k$ is such that $\rho$ restricted to the maximal unramified extension of $\mathbf Q_p$ in $K$ is the $k$-th power of the arithmetic Frobenius (this makes sense, because our assumption $K\subset L$ allows us to consider this restriction as an automorphism).\\
			By $e_\rho,f_\rho,h_\rho,$ and $z_\rho$ we will then denote the elements of $L\otimes_{\mathbf Q_p}\Lie G$ corresponding to the respective $1\otimes e,1\otimes f,1\otimes h,$ and $1\otimes z$ in $L\otimes_{\rho,K}\gl_2K$ via the decomposition. For each $x\in\{e,f,h,z\}$ we have $x_\rho=1_\rho\cdot(1\otimes x)$.\\ Furthermore, the decomposition extends to the universal enveloping algebra and its center: \begin{equation*}
				U_L\left(L\otimes_{\mathbf Q_p}\Lie G\right)\cong\bigotimes_{\rho\colon K\hookrightarrow L}U_L\left(L\otimes_{\rho,K}\gl_2K\right)
			\end{equation*} and \begin{equation*}
				\mathcal Z\left(U_L\left(L\otimes_{\mathbf Q_p}\Lie G\right)\right)\cong\bigotimes_{\rho\colon K\hookrightarrow L}\mathcal Z\left(U_L\left(L\otimes_{\rho,K}\gl_2K\right)\right).
			\end{equation*}
			Let us also note that we have a precise description of the center of the universal enveloping algebra: it is a polynomial algebra over $L$ generated by the $z_\rho$ and the \emph{Casimir elements}\begin{equation*}
				\Delta_\rho\coloneqq\frac12h_\rho^2+e_\rho f_\rho+f_\rho e_\rho=\frac12h_\rho^2+h_\rho+2f_\rho e_\rho.
			\end{equation*}\\
			Finally, we would like to make the above decomposition and the idempotents $1_\rho$ more explicit. In particular we would like to study their integrality properties.\\
			Note first that the natural injection \begin{equation}\label{unrdecomp}
				\mathcal O_L\otimes_{\mathbf Z_p}\mathcal O_K\hookrightarrow\bigoplus_{\rho\colon K\hookrightarrow L}\mathcal O_L\otimes_{\rho,\mathcal O_K}\mathcal O_K
			\end{equation} is only an isomorphism when $K$ is unramified over $\mathbf Q_p$. So the idempotents we have found are generally not contained in $\mathcal O_L\otimes_{\mathbf Z_p}\mathcal O_K$. We will make them more explicit in order to see how far they are from being in $\mathcal O_L\otimes_{\mathbf Z_p}\mathcal O_K$.\\
			First fix an $\alpha\in\mathcal O_K^\times$ so that $\mathbf Q_p[\alpha]$ is the maximal unramified subextension of $\mathbf Q_p$ in $K$. Then the corresponding decomposition for $\mathbf Q_p[\alpha]$ holds integrally, that is, $\mathcal O_L\otimes_{\mathbf Z_p}\mathbf Z_p[\alpha]$ is isomorphic to $\bigoplus_{i=0}^{f-1}\mathcal O_L\otimes_{\Frob^i,\mathbf Z_p[\alpha]}\mathbf Z_p[\alpha]$. We have corresponding idempotents \begin{equation*}
				\sum_{j=0}^{f-1}\Frob^i(\beta_j)\otimes\alpha^j\in\mathcal O_L\otimes_{\mathbf Z_p}\mathbf Z_p[\alpha], \text{ for $i\in\{0,\dots,f-1\}$},
			\end{equation*} with $\beta_j\in\mathbf Z_p[\alpha]^\times$.\\
			The extension $K/\mathbf Q_p[\alpha]$ is totally ramified and generated by $\varpi_K$, which is the root of an Eisenstein polynomial $f(X)\in\mathbf Q_p[\alpha][X]$. We only have the rational decomposition \begin{equation*}
				L\otimes_{\mathbf Q_p[\alpha]}K\cong\bigoplus_{\tau\colon K\hookrightarrow L}L\otimes_{\tau,K}K,
			\end{equation*} where the direct sum ranges over the $\mathbf Q_p[\alpha]$-linear embeddings $\tau$. We have corresponding idempotents \begin{gather}\label{ramdec}
				\begin{aligned}
					&1_{\id}=\sum_{i=0}^{e_K-1}\gamma_i\otimes\varpi_K^i\in L\otimes_{\mathbf Q_p[\alpha]}K,\\
					&1_\tau=\sum_{i=0}^{e_K-1}\tau(\gamma_i)\otimes\varpi_K^i\in L\otimes_{\mathbf Q_p[\alpha]}K,
				\end{aligned}
			\end{gather} with $\gamma_i\in K$. The non-integrality of the idempotents comes from the ramification. Thanks to the following lemma we can still control it sufficiently.
			\begin{lemma}\label{gammaialmostint}
				There is an $R_K\in\mathbf Z$ such that the $\gamma_i\in K$ (and therefore also their conjugates) defined by (\ref{ramdec}) have valuation \begin{equation*}
					v_p(\gamma_i)=-i/e_K-R_K/e_K.
				\end{equation*}
				Moreover, $R_K/e_K=v_p\left(f'(\varpi_K)\right)+1/e_K-1$, where $f$ is the minimal polynomial of $\varpi_K$ over $\mathbf Q_p[\alpha]$, and $R_K=0$ if and only if $p\nmid e_K$.
			\end{lemma}
			\begin{proof}
				We first briefly reflect on how to obtain the decomposition and the idempotents. This is done by identifying $\mathbf Q_p[\alpha][X]/f(X)$ with $K$, sending $X$ to $\varpi_K$. Then $L\otimes_{\mathbf Q_p[\alpha]}K\cong L[X]/f(X)$ and we can use the Chinese remainder theorem to get the decomposition, as $f(X)$ splits into coprime linear factors over $L$.\\
				The idempotent $1_\text{id}$ corresponds to the summand corresponding to the linear factor $X-\varpi_K$. Then one can check that, in terms of polynomials, we have \begin{equation*}
					1_\text{id}=\frac1{f'(\varpi_K)}\frac{f(X)}{X-\varpi_K}.
				\end{equation*}
				If $f(X)/(X-\varpi_K)=c_0+c_1X+\dots+c_{e_K-1}X^{e_K-1}$, then $\gamma_i=c_i/f'(\varpi_K)$. Since $(X-\varpi_K)\sum_ic_iX^i$ is Eisenstein over $\mathbf Z_p[\alpha]$ one can check inductively that $v_p(c_i)=1-(i+1)/e_K$. And indeed one then has $R_K/e_K=v_p\left(f'(\varpi_K)\right)+1/e_K-1$.\\
				Finally, $R_K=0$ if and only if $v_p\left(f'(\varpi_K)\right)=(e_K-1)v_p(\varpi_K)$, which is the case if and only if $v_p\left(\varpi_K-\sigma(\varpi_K)\right)=v_p(\varpi_K)$ for all automorphisms $\sigma\neq\id$ of $K$, which is the case if and only if the ramification is at most tame.
			\end{proof}
			Finally, to summarize, we have \begin{gather*}
				\begin{aligned}
					L\otimes_{\mathbf Q_p}K\cong L\otimes_{\mathbf Q_p}\mathbf Q_p[\alpha]\otimes_{\mathbf Q_p[\alpha]}K&\cong\left(\bigoplus_{i=0}^{f-1}L\otimes_{\Frob^i,\mathbf Q_p[\alpha]}\mathbf Q_p[\alpha]\right)\otimes_{\mathbf Q_p[\alpha]}K\\
					&\cong\bigoplus_{i=0}^{f-1}L\otimes_{\Frob^i,\mathbf Q_p[\alpha]}K\\
					&\cong\bigoplus_{i=0}^{f-1}\bigoplus_{\substack{\tau\colon K\hookrightarrow L,\\\tau|_{\mathbf Q_p[\alpha]=\Frob^i}}}L\otimes_{\tau,K}K\\&\cong\bigoplus_{\rho\colon K\hookrightarrow L}L\otimes_{\tau,K}K,
				\end{aligned}
			\end{gather*} where the last direct sum is indexed by all $\mathbf Q_p$-linear embeddings $\rho$.\\ We see that the idempotent of $L\otimes_{\mathbf Q_p}K$ corresponding to $\rho=\id$ is \begin{equation*}
				1_{\id}=\sum_{\substack{0\leq i<f,\\0\leq j<e_K}}\beta_i\gamma_j\otimes\alpha^i\varpi_K^j\in L\otimes_{\mathbf Q_p}K.
			\end{equation*} It follows that for a general $\mathbf Q_p$-linear embedding $\rho\colon K\hookrightarrow L$ one has \begin{equation}\label{idempotents}
				1_\rho=\sum_{\substack{0\leq i<f,\\0\leq j<e_K}}\rho(\beta_i\gamma_j)\otimes\alpha^i\varpi_K^j\in L\otimes_{\mathbf Q_p}K.
			\end{equation} 
			\subsection{Intermezzo: unramified $\GL_2K$ and the category $\mathcal C$}\label{interm}
			We jump a bit ahead by already considering the case that we initially studied for this work: for $\GL_2K$ with $K$ unramified and $p>2$. Ramification complicates the situation as we could already see in the previous section. We find that the proof for unramfied $\GL_2K$ is simpler and highlights well the main ideas of the general proof. We have therefore chosen to spend this brief section on the proof in this specific case. The reader who is only interested in the proof of the general case can safely skip this section.\\
			Our uniformizer for $K$ will be $\varpi_K=p$ in this section.
			\subsubsection*{Gelfand-Kirillov bound}
			\begin{theorem}\label{unrmainthm}
				Let $p>2$, suppose $K$ is unramified over $\mathbf Q_p$ of degree $f$, and suppose that we are in the $\GL_2$-case.\\
				Let $M_\circ$ be a finitely generated $p$-torsion free $\mathcal O_L\llbracket H\rrbracket$-module. Suppose that $M=M_\circ[1/p]$ has an infinitesimal character.\\
				Then $M$ has Gelfand-Kirillov dimension $\leq f$.
			\end{theorem}
			\begin{proof}
				We first collect the ingredients that go into the proof.
				\begin{enumerate}[label*=\alph*)]
					\item The ordered basis we choose for $H$ (with respect to which all the constructions are done) is that defined in (\ref{gl2orderedbasis}) in the next section. We simply denote it by $e_i,f_i,h_i,z_i$, for $0\leq i<f$.
					\item The $\gr_{r_{N'}}\mathcal O_L\llbracket H^{p^N}\rrbracket$, for $N\geq0$, are polynomial algebras in the $\mathbf e_k^{p^N},\mathbf f_k^{p^N},\mathbf h_k^{p^N},\mathbf z_k^{p^N}$, $0\leq k<f$, given by \begin{equation*}
						\mathbf x_k^{p^N}=\sum_{i=0}^{f-1}\Frob^{k+N}\beta_i\cdot\sigma_{r_{N'}}\left(\left[x_i^{p^N}\right]-1\right),
					\end{equation*} where $x\in\{e,f,h,z\}$ is considered sensitive to \verb|\mathbf|.\\
					This is Lemma \ref{commutativity} (together with the decomposition (\ref{unrdecomp}) of the residue fields) applied to the group $H^{1/p}$.
					\item The $p^{2N+6}\Delta_{\Frob^k}$, considered as elements of $D_{1/p}(H^{p^N},L)$, have principal symbol \begin{equation}\label{unrcasimirmultiplied}
						\sigma_{1/p}\left(p^{2N+6}\Delta_{\Frob^k}\right)= \frac12\left(\mathbf h_{k-N}^{p^N}\right)^2+2\varepsilon\mathbf e_{k-N}^{p^N}\mathbf f_{k-N}^{p^N}.
					\end{equation}
					This is Lemma \ref{lieelts} of two sections ahead.
				\end{enumerate}
				Now we choose and fix, as usual, a surjection $\mathcal O_L\llbracket H\rrbracket^{\oplus N_0}\twoheadrightarrow M_\circ$ so that we can define all the filtrations on $M_\circ$. Let us denote the infinitesimal character by $\lambda$ and take $N$ sufficiently large so that all the $\lambda\left(p^{2N+6}\Delta_{\Frob^k}\right)$ have valuation bigger than the valuation of (\ref{unrcasimirmultiplied}) ($\geq2p/(p-1)$ is clearly large enough). Then the \begin{equation*}
					p^{2N+6}\Delta_{\Frob^k}-\lambda\left(p^{2N+6}\Delta_{\Frob^k}\right)
				\end{equation*} act by $0$ on $M_{r_N}$ and their principal symbols are still as in (\ref{unrcasimirmultiplied}).\\
				Take $N$ possibly even larger, to ensure that $\gr_{r_N}M_\circ\to\gr_{r_N}M_{r_N}$ is an injection, possible by Corollary \ref{injectionsforNsuffbig}.\\
				Then the principal symbols in (\ref{unrcasimirmultiplied}) all annihilate $\gr_{r_N}M_\circ$. This means that the $\overline{\mathbf h}_k$ act nilpotently on $\gr_{r_N}M_\circ/\varepsilon_L\gr_{r_N}M_\circ$. By $\varepsilon$-torsion freeness of $\gr_{r_N}M_\circ$ and Lemma \ref{modeps=grmodp}, $\gr_{r_N}M_\circ/\varepsilon_L\gr_{r_N}M_\circ$ actually equals $\gr\overline M$, which does not depend on $N$. Therefore we can repick $N$ (at least as large as before) to find that multiplication by the $\mathbf h_k^{p^N}$ sends $\gr_{r_N}M_\circ$ into $\varepsilon\gr_{r_N}M_\circ$. The $\frac12\left(\mathbf h_k^{p^N}\right)^2/\varepsilon$ then also send $\gr_{r_N}M_\circ$ into $\varepsilon\gr_{r_N}M_\circ$. At the same time, all \begin{equation*}
					\frac12\left(\mathbf h_{k-N}^{p^N}\right)^2/\varepsilon+2\mathbf e_{k-N}^{p^N}\mathbf f_{k-N}^{p^N}
				\end{equation*} act by $0$ on $\gr_{r_N}M_\circ$, by its $\varepsilon$-torsion freeness. So we conclude that all the $\overline{\mathbf e}_k\overline{\mathbf f}_k$ act nilpotently on $\gr\overline M$.\\
				Finally, also the $\overline{\mathbf z}_k$ act nilpotently on $\gr\overline M$, because \begin{equation*}
					\sigma_{1/p}\left(p^{N+3}z_{\Frob^k}\right)=\mathbf z_{k-N}^{p^N}\in\gr_{1/p}D_{1/p}(H^{p^N},L)
				\end{equation*} where the $p^{N+3}z_{\Frob^k}$ are considered elements of $D_{1/p}(H^{p^N},L)$ -- again Lemma \ref{lieelts}. Choosing $N$ possibly even larger one finds the $\lambda\left(p^{N+3}z_{\Frob^k}\right)$ to have positive valuation and concludes as before that the $\overline{\mathbf z}_k$ act nilpotently on $\gr\overline M$.\\
				Since the Krull dimension of $\gr k_L\llbracket H\rrbracket/\left(\overline{\mathbf h}_k,\overline{\mathbf e}_k\overline{\mathbf f}_k,\overline{\mathbf z}_k\colon0\leq k<f\right)$ is $f$, it follows that $\gr\overline M$ has Krull dimension at most $f$ and this gives us the desired bound for the Gelfand-Kirillov dimension, by the discussion at the end of Section \ref{GKsection}.
			\end{proof}
			\subsubsection*{The category $\mathcal C$}
			One of the original motivations for studying the annihilation properties as above is a category -- which we call $\mathcal C$ -- of smooth admissible mod $k_L$-representations of $\GL_2\mathbf Q_{p^f}=\GL_2K$ first studied in \cite[p.~113]{BHHMSConjs} and \cite[\S5.3]{BHHMS1}. This is the category of those smooth admissible mod $k_L$-representations of $\GL_2\mathbf Q_{p^f}$ that have a central character and for which the associated graded $\gr k_L\llbracket I_1\rrbracket$-module $\gr\pi^\vee$ is killed by a power of the two-sided ideal \begin{equation*}
				I_G\coloneqq\left(\overline\sigma(ph_{\Frob^i}),\overline\sigma(e_{\Frob^i})\overline\sigma(pf_{\Frob^i}))\colon0\leq i<f\right)\subset U_{k_L}\left(\overline{\gr I_1}\right)\cong\gr k_L\llbracket I_1\rrbracket,
			\end{equation*} where $I_1=H^{1/p^2}$ is the upper triangular pro-$p$ Iwahori subgroup of $\GL_2K$.\\
			The central character condition ensures, moreover, that $\gr\pi^\vee$ is killed by the $\overline\sigma(pz_{\Frob^i})$, since these elements of $\gr k_L\llbracket I_1\rrbracket$ are (principal symbols of) linear combinations of elements $z-1$ with $z\in I_1$ central, and the center of $I_1$ acts trivially on $\pi$.\\
			It follows that for such $\pi$ the associated graded $\gr\pi^\vee$ is in fact killed by a power of the two-sided ideal \begin{equation*}
				J\coloneqq\left(\overline\sigma(ph_{\Frob^i}),\overline\sigma(e_{\Frob^i})\overline\sigma(pf_{\Frob^i}),\overline\sigma(pz_{\Frob^i})\colon0\leq i<f\right).
			\end{equation*}
			There is a clear relation with what we saw in Theorem \ref{unrmainthm}. First of all, the ideal $J$ contains the ideal \begin{equation*}
				\left(\overline{\mathbf h}_k,\overline{\mathbf e}_k\overline{\mathbf f}_k,\overline{\mathbf z}_k\colon0\leq k<f\right)\subset\gr k_L\llbracket I_1^{p^2}\rrbracket
			\end{equation*} of the proof of the theorem. So a power of this latter ideal kills $\gr\pi^\vee$ if a power of $J$ does so. Like in the proof of our theorem this gives a bound on the Gelfand-Kirillov dimension: representations $\pi$ in $\mathcal C$ have Gelfand-Kirillov dimension at most $f$ (see \cite[Cor.~5.3.5]{BHHMS1}).\\
			In the other direction, we can start with an admissible $L$-Banach representation $\Pi$ with infinitesimal character of $\GL_2K$ which is moreover unitary, so that it has an open unit ball which is $\GL_2K$-stable and of which the mod $\varpi_L$-reduction gives us then a smooth admissible $k_L$-representation $\pi$. If $M$ is the $L$-dual of $\Pi$, then $\pi^\vee$ arises as $\overline M$ for some finitely generated $\mathcal O_L\llbracket I_1\rrbracket$-module $M_\circ$ with $M_\circ[1/p]=M$. So the proof of Theorem \ref{unrmainthm} tells us $\gr\pi^\vee$ is killed by a power of $\left(\overline{\mathbf h}_k,\overline{\mathbf e}_k\overline{\mathbf f}_k,\overline{\mathbf z}_k\colon0\leq k<f\right)$.\\
			If we assume that $\pi$ moreover has a central character, then \cite{sorgdrager2025notecertaincategorymod} implies that $\gr\pi^\vee$ is actually killed by a power of $I_G$, so that $\pi$ will be in $\mathcal C$.\\
			
			The category $\mathcal C$ still has a mysterious but seemingly important role in the mod $p$ Langlands program and was before studied purely in a mod $p$ setting. We hope that the ideas going into Theorem \ref{unrmainthm} can shine some more light on it. As \cite{sorgdrager2025notecertaincategorymod} shows that the category can be defined with respect to arbitrarily small $I_1^{p^N}$, the more general Theorem \ref{mainmainthm} of below seems to point to the ideals that might define similar categories of mod $p$ representations for groups like $\GL_2K$ with $K$ ramified or for units of quaternions.\\
			Indeed, using the results in \cite{sorgdrager2025notecertaincategorymod} again, Theorem \ref{mainmainthm} specialized to quaternion units over an unramified $K$ also recovers the category of mod $p$ representations studied in \cite{hu2024modprepresentationsquaternion} in the same way as it recovers $\mathcal C$ when specialized to unramified $\GL_2$.	
			\subsection{Ordered basis and generators of the Lie algebra $\mathcal O_L\otimes_{\mathbf Z_p}H$}
			We will now choose an ordered basis for $H$, although we will not really care about ordering it and we will leave this to the reader.\\ In the $\GL_2$-case it will consist of the $e_{i,j}, f_{i,j},h_{i,j},z_{i,j}$ for $0\leq i<f, 0\leq j<e_K$, defined as follows: for $x\in\{e,f,h,z\}$ we set \begin{equation}\label{gl2orderedbasis}
				x_{i,j}\coloneqq\begin{cases}
					\exp(\alpha^i\varpi_K^jp^2e)=\begin{pmatrix}
						1&\alpha^i\varpi_K^jp^2\\&1
					\end{pmatrix}&\text{if $x=e$,}\\
					\exp(\alpha^i\varpi_K^{j+1}p^2x)&\text{else.}
				\end{cases}
			\end{equation} Using the logarithm $\log\colon H\to M_2\mathcal O_K$ one checks that these indeed define an ordered basis.\\
			In the quaternion case we pick the following ordered basis for $H=1+p^2\Pi\mathcal O_D$:
			\begin{gather*}
				\begin{aligned}
					&a_{i,j}\coloneqq\exp(p^2\alpha^i\Pi^{2j+1}),
					b_{i,j}\coloneqq\exp(p^2\sqrt a\alpha^i\Pi^{2j+1}),\\
					&c_{i,j}\coloneqq\exp(p^2\sqrt a\alpha^i\varpi_K^{1+j}),
					z_{i,j}\coloneqq\exp(p^2\alpha^i\varpi_K^{1+j}),
				\end{aligned}
			\end{gather*}where $0\leq i<f, 0\leq j<e_K$.	Again, it can be seen using the logarithm that they constitute an ordered basis.\\
			Regarding $H$ as a Lie algebra over $\mathbf Z_p$ now, we find, basically by definition, that it is generated over $\mathbf Z_p$ by the ordered basis. Using that $\mathcal O_L\otimes_{\mathbf Z_p}\mathcal O_K$ is generated over $\mathcal O_L$ by the $\sum_i\Frob^k(\beta_i)\otimes\alpha^i$ and the $1\otimes\varpi_K^j$, and motivated in the quaternion case by (\ref{quatliealg}) we observe that the $\mathcal O_L$-Lie algebra $\mathcal O_L\otimes_{\mathbf Z_p}H$ has the following generators over $\mathcal O_L$:
			\begin{itemize}
				\item the $\sum_{i=0}^{f-1}\Frob^k\beta_i\otimes x_{i,j}$ for $0\leq k<f, 0\leq j<e_K,$ and $x\in\{e,f,h,z\}$ in the $\GL_2$-case;
				\item in the quaternion case we find the following generators \begin{gather*}
					\begin{aligned}
						&\frac12\left(\sum_{i=0}^{f-1}\Frob^k\beta_i\otimes a_{i,j}+\Frob^k\left(\frac{\beta_i}{\sqrt a}\right)\otimes b_{i,j}\right),\\
						&\frac12\left(\sum_{i=0}^{f-1}\Frob^k\beta_i\otimes a_{i,j}-\Frob^k\left(\frac{\beta_i}{\sqrt a}\right)\otimes b_{i,j}\right),\\
						&\sum_{i=0}^{f-1}\Frob^k\left(\frac{\beta_i}{\sqrt a}\right)\otimes c_{i,j}, \text{ and }\sum_{k=0}^{f-1}\Frob^k\beta_i\otimes z_{i,j},
					\end{aligned}
				\end{gather*} where $0\leq k<f, 0\leq j<e_K$.
			\end{itemize}
			By multiplying the above elements by $p^N$ (which is equivalent to multiplying the arguments of the exponentials in the expressions by $p^N$), we obtain generators over $\mathcal O_L$ of $\mathcal O_L\otimes_{\mathbf Z_p}H^{p^N}=\mathcal O_L\otimes_{\mathbf Z_p}p^N\cdot H$.\\
			
			Finally, let us note that now that we have an ordered basis on $H$ we can consider the distribution algebras $D_{r_n}(H,L)$ etc.~and in particular the corresponding filtrations and whenever we will do so it will be respect to the choses basis (with the prefered order of the reader).
			\subsection{Congruences}
			We first describe the rings $\gr_{r_{N'}}\mathcal O_L\llbracket H^{p^N}\rrbracket$ in a way that is more amenable to our Lie elements.
			\begin{lemma}
				Let $N,N'\geq0$. The $k_L$-algebra $\gr_{r_{N'}}\mathcal O_L\llbracket H^{p^N}\rrbracket$ is polynomial with generators $\varepsilon_L$ and \begin{gather*}
					\begin{aligned}
						\mathbf x_{k,j}^{(p^N)}\coloneqq\sigma_{r_{N'}}\left(\sum_{i=0}^{f-1}\Frob^{k+N}\beta_i\left(\left[x_{i,j}^{p^N}\right]-1\right)\right), x\in\{e,f,h,z\}, 0\leq k<f, 0\leq j<e_K
					\end{aligned}
				\end{gather*} in the $\GL_2$-case and by $\varepsilon_L$ and
				\begin{gather*}
					\begin{aligned}
						\mathbf w_{k,j}^{(p^N)}&\coloneqq\frac12\sigma_{r_{N'}}\left(\sum_{i=0}^{f-1}\Frob^{k+N}\beta_i\left(\left[a_{i,j}^{p^N}\right]-1\right)+\Frob^{k+N}\left(\frac{\beta_i}{\sqrt a}\right)\left(\left[b_{i,j}^{p^N}\right]-1\right)\right),\\
						\mathbf h_{\ell,j}^{(p^N)}&\coloneqq\sigma_{r_{N'}}\left(\sum_{i=0}^{f-1}\Frob^{\ell+N}\left(\frac{\beta_i}{\sqrt a}\right)\left(\left[c_{i,j}^{p^N}\right]-1\right)\right),\\
						\mathbf z_{\ell,j}^{(p^N)}&\coloneqq\sigma_{r_{N'}}\left(\sum_{i=0}^{f-1}\Frob^{\ell+N}\beta_i\left(\left[z_{i,j}^{p^N}\right]-1\right)\right), 0\leq k<2f, 0\leq\ell<f, 0\leq j<e_K
					\end{aligned}
				\end{gather*} in the quaternion case.
			\end{lemma}
			\begin{remark}\label{abusnotrem}
				The notation ``$\mathbf x_{k,j}^{(p^N)}$'', $\mathbf x\in\{\mathbf e, \mathbf f,\mathbf h,\mathbf z,\mathbf w\}$, defined above omits the radius of analyticity $r_{N'}$ and thus leads to the abuse of notation where countably infinitely many different elements (different because they live in different rings) are denoted in the same way. In order to avoid confusion, we will always make sure the ambient ring will be clear.\\
				The reductions modulo $\varepsilon_L$ of these elements, denoted $\overline{\mathbf x}^{(p^N)}_{k,j}$, are however always the same element in $\gr k_L\llbracket H^{p^N}\rrbracket\cong\gr_{r_{N'}}k_L\llbracket H^{p^N}\rrbracket$.
			\end{remark}
			We have the following compatibility with respect to $p$-powers.
			\begin{lemma}
				Let $N,N'\geq0$ be integers and treat the injection\begin{equation*}
					\gr_{r_{N'}}\mathcal O_L\llbracket H^{p^N}\rrbracket\hookrightarrow\gr_{r_{N+N'}}\mathcal O_L\llbracket H\rrbracket
				\end{equation*} as an inclusion.
				In the $\GL_2$-case we then have \begin{equation*}
					\mathbf x_{k,j}^{(p^N)}=\mathbf x_{k,j}^{p^N},
				\end{equation*} where $\mathbf x_{k,j}\coloneqq\mathbf x_{k,j}^{(p^0)}$, and similarly in the quaternion case. For this reason we will omit the parentheses in the notations defined in the previous lemma.
			\end{lemma}
			\begin{proof}
				This is Lemma \ref{compatibilityppowers}.
			\end{proof}
			Finally, the relation with $\Lie H$ is described by the next lemma.
			\begin{lemma}\label{lieelts}
				Consider the natural map \begin{equation*}
					L\otimes_{\mathbf Q_p}\gl_2K\cong L\otimes_{\mathbf Q_p}\Lie H\to D_{r_N}(H,L).
				\end{equation*}		
				Considering all principal symbols to be taken in $\gr_{r_N}D_{r_N}(H,L)$ we have: \begin{gather*}
					\begin{aligned}
						\sigma_{r_N}(e_\rho)&=\begin{cases}
							\frac1{\varepsilon^{N+2}}\sum_{j=0}^{e_K-1}\rho(\gamma_j)\mathbf e_{k-N,j}^{p^N}&\text{$\GL_2$-case,}\\[1ex]
							\frac1{\varepsilon_K\varepsilon^{N+2}}\sum_{j=0}^{e_K-1}\rho(\gamma_j)\mathbf w_{k-N,j}^{p^N}&\text{quaternion case,}
						\end{cases}\\
						\sigma_{r_N}(f_\rho)&=\begin{cases}
							\frac1{\varepsilon_K\varepsilon^{N+2}}\sum_{j=0}^{e_K-1}\rho(\gamma_j)\mathbf f_{k-N,j}^{p^N}&\text{$\GL_2$-case,}\\[1ex]
							\frac1{\varepsilon^{N+2}}\sum_{j=0}^{e_K-1}\rho(\gamma_j)\mathbf w_{k+f-N,j}^{p^N}&\text{quaternion case,}
						\end{cases}\\
						\sigma_{r_N}(h_\rho)&=\frac1{\varepsilon_K\varepsilon^{N+2}}\sum_{j=0}^{e_K-1}\rho(\gamma_j)\mathbf h_{k-N,j}^{p^N},\\
						\sigma_{r_N}(z_\rho)&=\frac1{\varepsilon_K\varepsilon^{N+2}}\sum_{j=0}^{e_K-1}\rho(\gamma_j)\mathbf z_{k-N,j}^{p^N},
					\end{aligned}
				\end{gather*}
				where $\rho$ restricts to $\Frob^k$.\\
				The Casimir operators have the following principal symbols\begin{equation*}
					\sigma_{r_N}\left(\Delta_\rho\right)=
					\frac1{2\varepsilon_K^2\varepsilon^{2N+4}}\left(\sum_{j=0}^{e_K-1}\rho(\gamma_j)\mathbf h_{k-N,j}^{p^N}\right)^2+\frac2{\varepsilon_K\varepsilon^{2N+4}}\left(\sum_{j=0}^{e_K-1}\rho(\gamma_j)\mathbf e_{k-N,j}^{p^N}\right)\cdot\left(\sum_{j=0}^{e_K-1}\rho(\gamma_j)\mathbf f_{k-N,j}^{p^N}\right)
				\end{equation*} in the $\GL_2$-case and 
				\begin{equation*}
					\sigma_{r_N}\left(\Delta_\rho\right)=\frac1{2\varepsilon_K^2\varepsilon^{2N+4}}\left(\sum_{j=0}^{e_K-1}\rho(\gamma_j)\mathbf h_{k-N,j}^{p^N}\right)^2+\frac2{\varepsilon_K\varepsilon^{2N+4}}\left(\sum_{j=0}^{e_K-1}\rho(\gamma_j)\mathbf w_{k-N,j}^{p^N}\right)\left(\sum_{j=0}^{e_K-1}\rho(\gamma_j)\mathbf w_{k+f-N,j}^{p^N}\right)
				\end{equation*} in the quaternion case.
			\end{lemma}
			\begin{proof}
				First we note that in general, when $\mathbf r\in\Lie G$, $\exp\mathbf r$ exists, and the valuation of $[\exp\mathbf r]-1$ is $>1/(p-1)$ then the image of $\mathbf r$ in the distribution algebra, which is \begin{equation*}
					-\sum_{i=1}^\infty\frac{\left(1-\left[\exp\mathbf r\right]\right)^i}i
				\end{equation*} (as discussed around (\ref{psi})), has the same principal symbol as $[\exp\mathbf r]-1$.\\
				From this and the decomposition of the Lie algebra (\ref{Liedecomp}) and in particular (\ref{idempotents}) (also (\ref{quatliealg}) in the quaternion case) the expressions for the principal symbols of $e_\rho,f_\rho,h_\rho,z_\rho$ already follow, provided that the taking of principal symbols is additive with respect to our sums. Since the expressions that we claim to be the principal symbols are clearly non-zero, it suffices to check they are homogenous. This follows from Lemma \ref{gammaialmostint} and the fact that the valuation of an $\mathbf x^{p^N}_{i,j+1}$ is always precisely $1/e_K$ more than that of $\mathbf x^{p^N}_{i,j}$ for $0<j+1<e_K$ (which follows from the definition of the $p$-valuation on $H$).\\
				The formulas for the Casimir operators might then seem immediate, but again there is the subtlety of checking homogenity which we leave to the reader this time.
			\end{proof}
			
			\subsection{Bound in the general case}
			\begin{theorem}\label{mainmainthm}
				Let $p>2$, let $K$ be a $p$-adic field, and let $D$ be the quaternion algebra over $K$. If $\Pi$ is an admissible $L$-Banach representation with infinitesimal character of $\GL_2K$ or $D^\times$, then its Gelfand-Kirillov dimension is $\leq[K\colon\mathbf Q_p]$.
			\end{theorem}
			\begin{proof}
				Choose and fix an embedding $\rho\colon K\hookrightarrow L$. And let $0\leq k<f$ such that $\rho|_{\mathbf Q_p[\alpha]}=\Frob^k$.\\ In the light of Lemma \ref{gammaialmostint} we define \begin{equation*}
					\mu_{j,\rho}\coloneqq\varpi_K^{R_K+j}\rho(\gamma_j)\in\mathcal O_{K'}^\times\end{equation*} and we denote by $\overline\mu_{j,\rho}\in k_{K'}=k_K$ their reductions modulo $\varpi_{K'}$.\\
				Let $C$ be either the quadratic Casimir operator $\Delta_\rho$ or the central Lie algebra element $z_\rho$. For $N\in f\mathbf Z_{\geq0}$ consider the elements\begin{equation*}
					C_N\coloneqq\begin{cases}
						\left(p^{N+2}\varpi_K^{e_K+R_K}\right)^2C&\text{if }C=\Delta_\rho,\\
						p^{N+2}\varpi_K^{e_K+R_K}C&\text{if }C=z_\rho,
					\end{cases}
				\end{equation*} of $\mathcal Z\left(U(L\otimes_{\mathbf Q_p}\gl_2K)\right)$.\\
				Then, by Lemma \ref{lieelts}, $\sigma_{r_N}(C_N)\in\gr_{r_N}\mathcal O_L\llbracket H\rrbracket\subset\gr_{r_N}D_{r_N}(H,L)$ and it equals \begin{equation}\label{casimirtobeexpanded}\begin{gathered}
					\frac12\left(\sum_{j=0}^{e_K-1}\overline\mu_{e_K-1-j,\rho}\mathbf h_{k,e_K-1-j}^{p^N}\varepsilon_K^j\right)^2\\+2\varepsilon_K\left(\sum_{j=0}^{e_K-1}\overline\mu_{e_K-1-j,\rho}\mathbf e_{k,e_K-1-j}^{p^N}\varepsilon_K^j\right)\cdot\left(\sum_{j=0}^{e_K-1}\overline\mu_{e_K-1-j,\rho}\mathbf f_{k,e_K-1-j}^{p^N}\varepsilon_K^j\right)\end{gathered}
				\end{equation} in the case $C=\Delta_\rho$ and of $\GL_2$ and it equals \begin{equation*}\label{zcoeffs}
					\sum_{j=0}^{e_K-1}\overline\mu_{e_K-1-j,\rho}\mathbf z_{k,e_K-1-j}^{p^N}\varepsilon_K^j
				\end{equation*} in the case $C=z_\rho$. In the quaternion case one has to replace the $\mathbf e_{k,e_K-1-j}$ by $\mathbf w_{k,e_K-1-j}$ and the $\mathbf f_{k,e_K-1-j}$ by $\mathbf w_{k+f,e_K-1-j}$ to get the right expression in the quadratic Casimir case.\\
				(By expanding out the brackets in the quadratic Casimir case,) we find in both cases that $\sigma_{r_N}(C_N)$ is a polynomial in $\varepsilon_K$:\begin{equation}\label{polyexpansion}
					\sigma_{r_N}(C_N)\eqqcolon\sum_{i=0}^{d_C}c_{N,i}\varepsilon_K^i,
				\end{equation} where the degree equals $d_C=\begin{cases}
					2e_K-1&\text{if }C=\Delta_\rho,\\e_K-1&\text{if }C=z_\rho,
				\end{cases}$ and where the $C_{N,i}\in\gr_{r_N}\mathcal O_L\llbracket H\rrbracket$ are homogeneous polynomial expressions over $k_{K'}=k_K$ in the variables $\mathbf x_{k,j}$, with $\mathbf x\in\{\mathbf e, \mathbf f, \mathbf h, \mathbf w,\mathbf z\}$.\\
				
				\textbf{Observation:} for all $0\leq i\leq d_C$ and all $m\in\mathbf Z_{\geq0}$ we have \begin{equation*}
					\overline c_{N,i}^{p^{mf}}=\overline c_{N+mf,i}\text{ in }\gr k_L\llbracket H\rrbracket
				\end{equation*} and actually $\overline c_{N,i}\in\gr k_L\llbracket H\rrbracket\cong\gr_{r_{N+mf}}\mathcal O_L\llbracket H\rrbracket/\varepsilon_L$ lifts\footnote{Indeed, such lift in $\gr_{r_{N+mf}}\mathcal O_L\llbracket H\rrbracket$ can simply be given by the same formula as $C_{N,i}\in\gr_{r_N}\mathcal O_L\llbracket H\rrbracket$ in terms of the variables $\mathbf x_{k,j}$, but where these variables are now interpreted as elements of $\gr_{r_{N+mf}}\mathcal O_L\llbracket H\rrbracket$.} to a $p^{mf}$-th root of $c_{N,i}$ in $\gr_{r_{N+mf}}\mathcal O_L\llbracket H\rrbracket$.\\
				
				Now choose $s\in\mathbf Z_{>0}$ minimal such that $\log\left(p^{sf}\right)\geq e_L/e_K$. Then we make the following claim.\\
				
				\textbf{Claim:} let $M_\circ$ be a finitely generated $\mathcal O_L\llbracket H\rrbracket$-module and fix an $\mathcal O_L\llbracket H\rrbracket^{\oplus N_0}\twoheadrightarrow M_\circ$, inducing on it all $\Fil_{r_N}^\bullet$-filtrations as usual. Suppose that there is an $N\in f\mathbf Z_{\geq0}$ such that all $\gr_{r_{N+sfi}}M_\circ$ are $\varepsilon$-torsion free and killed by $\sigma_{r_{N+sfi}}(C_{N+sfi})$ for $0\leq i\leq d_C$. Then the $\overline c_{N+sfi,i}$ annihilate $\gr\overline M$ for $0\leq i\leq d_C$.\\
				
				We prove the claim by induction on $i$. We first observe that the $\varepsilon$-torsion freeness assumption implies the $\gr_{r_{N+sfi}}M_\circ/\varepsilon_L\gr_{r_{N+sfi}}M_\circ$ all equal $\gr\overline M$, by Lemma \ref{modeps=grmodp}. So the annihilation of $\gr\overline M$ by $\overline c_{N,0}$ is a direct consequence of the assumptions and $\sigma_{r_N}(C_N)\mod\varepsilon_L\equiv\overline c_{N,0}$.\\
				Now suppose the $\overline c_{N+sf\ell,\ell}$ annihilate $\gr\overline M$ for $0\leq \ell<i$. Then, by the observation above, multiplication by $c_{N+sfi,\ell}$ defines a map \begin{equation*}
					\gr_{r_{N+sf\ell}}M_\circ\to\varepsilon_L^{p^{sf(i-\ell)}}\gr_{r_{N+sf\ell}}M_\circ,
				\end{equation*} which means that multiplication by $c_{N+sfi,\ell}\varepsilon_K^\ell$ defines a map 
				\begin{equation*}
					\gr_{r_{N+sf\ell}}M_\circ\to\varepsilon_L\varepsilon_K^i\gr_{r_{N+sf\ell}}M_\circ,
				\end{equation*} using that \begin{equation*}
					p^{sf(i-\ell)}/e_L+\ell/e_K>\log(p^{sf})(i-\ell)/e_L+\ell/e_K\geq i/e_K.
				\end{equation*}
				As a result, the element \begin{equation*}
					t_i\coloneqq\frac1{\varepsilon_K^i}\sum_{\ell=0}^{i-1}c_{N+sfi,\ell}\varepsilon_K^\ell\in\gr_{r_{N+sf\ell}}L\llbracket H\rrbracket
				\end{equation*} stabilizes  $\gr_{r_{N+sf\ell}}M_\circ\subset\gr_{r_{N+sf\ell}}M_\circ[1/\varepsilon]$ and acts trivially on its quotient $\gr\overline M$.\\
				By assumption, the element $\sigma_{r_{N+sfi}}(C_N)/\varepsilon_K^i\in\gr_{r_{N+sf\ell}}L\llbracket H\rrbracket$ annihilates $\gr_{r_{N+sf\ell}}M_\circ\subset\gr_{r_{N+sf\ell}}M_\circ[1/\varepsilon]$. Consequently, $\sigma_{r_{N+sfi}}(C_N)/\varepsilon_K^i-t_i$ annihilates $\gr\overline M$. At the same time $\sigma_{r_{N+sfi}}(C_N)/\varepsilon_K^i$ is an element of $\gr_{r_{N+sf\ell}}\mathcal O_L\llbracket H\rrbracket$ with reduction modulo $\varepsilon_L$ equal to $\overline c_{N+sfi,i}$, by (\ref{polyexpansion}). This finishes the proof of the claim.\\
				
				We have done sufficient preparations now to prove the theorem. Let $M$ be $\Pi^*$ and take a finitely generated sub-$\mathcal O_L\llbracket H\rrbracket$-module $M_\circ$ inside $M$ so that $M_\circ[1/p]=M$.\\
				Fix $\mathcal O_L\llbracket H\rrbracket^{\oplus N_0}\twoheadrightarrow M_\circ$ as in the claim above.  For $N\gg0$ the conditions of the above claim are satisfied. Indeed, choosing $N\gg0$ we can first of all ensure that all $\gr_{r_{N'}}M_\circ$ are $\varepsilon$-torsion free for all $N'\geq N$, by Proposition \ref{epstorsfree}. Secondly, observe that the degree of $\sigma_{r_{N'}}(C_{N'})\in\gr_{r_{N'}}\mathcal O_L\llbracket H\rrbracket$ is independent of $N'$, while the degree of $\sigma_{r_{N'}}\left(\lambda(C_{N'})\right)$ is the $p$-adic valuation of $\lambda(C_{N'})$, which tends to $\infty$ as $N'\to\infty$. Therefore, \begin{equation*}
					\sigma_{r_{N'}}\left(C_{N'}-\lambda(C_{N'})\right)=\sigma_{r_{N'}}(C_{N'})\text{ for $N'\gg0$}.
				\end{equation*} 
				Since the left hand side of the above equation acts by $0$ on $\gr_{r_{N'}}M_{r_{N'}}$ and since Lemma \ref{localizationtoeps} tells us $\gr_{r_{N'}}M_\circ[1/\varepsilon]=\gr_{r_{N'}}M_{r_{N'}}$, we see that by choosing $N$ possibly larger the conditions of the above claim can indeed be met.\\
				
				Finally, choose for each $0\leq k<f$ an embedding $\rho_k\colon K\hookrightarrow L$ restricting to $\Frob^k$. Define, for $0\leq i<2e_K$, $\overline c_{\rho_k,i}\in\gr k_L\llbracket H\rrbracket$ to be the coefficient $\overline c_{0,i}$ in the case $C=\Delta_{\rho_k}$. Resuming all that was done above, we find that a power of the ideal \begin{equation}\label{casimirideal}
					I\coloneqq\left(\overline c_{\rho_k,i}, \overline{\mathbf z}_{k,i}\colon 0\leq i<2e_K,0\leq k<f\right)\subset\gr k_L\llbracket H\rrbracket
				\end{equation} annihilates $\gr\overline M$.\\
				Therefore, the Krull dimension of $\gr\overline M$, which is the Gelfand-Kirillov dimension of $\Pi$, is at most the Krull dimension of the ring $\gr k_L\llbracket H\rrbracket/I$, which we claim is at most $[K\colon\mathbf Q_p]$. To see this, first observe that \begin{equation*}
					\gr k_L\llbracket H\rrbracket/I\cong\otimes_{k=0}^{f-1}R_k/I_k
				\end{equation*}
				where the tensor product is over $k_L$ and where, for $0\leq k<f$, \begin{equation*}
					R_k\coloneqq\begin{cases}
						k_L[\overline{\mathbf e}_{k,i},\overline{\mathbf f}_{k,i},\overline{\mathbf h}_{k,i}\colon0\leq i<e_K]&\text{$\GL_2$-case,}\\
						k_L[\overline{\mathbf w}_{k,i},\overline{\mathbf w}_{k+f,i},\overline{\mathbf h}_{k,i}\colon0\leq i<e_K]&\text{quaternion case,}
					\end{cases}
				\end{equation*} is a polynomial algebra over $k_L$ with ideal \begin{equation*}
					I_k\coloneqq(\overline c_{\rho_k,i}\colon0\leq i<e_K)\subset R_k.
				\end{equation*}
				It therefore suffices to show $R_k/I_k$ has Krull dimension at most $e_K$. This follows from the Lemma \ref{dimensioncalcul} below, because $R_k/I_k$ is abstractly isomorphic to the ring of this lemma for $n=e_K-1$, which can be seen by (\ref{casimirtobeexpanded}) and the fact that the $\overline\mu_{\rho_k,j}$ are non-zero.
			\end{proof}
			\begin{lemma}\label{dimensioncalcul}
				Let $n\geq0$ and $R\coloneqq{k_L}[u_0,\dots,u_n,v_0,\dots,v_n,w_0,\dots,w_n]$ and $U(X)=\sum_{i=0}^nu_iX^i, V(X)=\sum_{i=0}^nv_iX^i, W(X)=\sum_{i=0}^nw_iX^i\in R[X]$.\\
				Let $I$ be the ideal of $R$ generated by the coefficients of $U(X)^2-V(X)W(X)X$. Then $R/I$ has Krull dimension at most $n+1$.
			\end{lemma}
			\begin{proof}
				Let $\mathfrak p\subset R$ be a prime ideal containing $I$. It suffices to show $R/\mathfrak p$ has Krull dimension at most $n+1$. Let $F$ be the fraction field of $R/\mathfrak p$ and denote by $q\colon R\to F$ the obvious map.\\
				Since $U(X)^2=V(X)W(X)X$ over $F$, $X$ has to divide $U(X)$ over $F$ and then also one of $V(X)$ and $W(X)$. We obtain the following identity over $F$:\begin{equation}\label{simplifiedeq}
					\tilde U(X)^2=\tilde V(X)\tilde W(X)
				\end{equation}
				where $\tilde U(X)=q\left(U(X)\right)/X$ and one of the $\tilde V(X)$ and $\tilde W(X)$ similarly differs from the original polynomial. Denote by $d$ the degree of $\tilde U$, so that $d<n$. Let $a$ and $b$ denote the degrees of $\tilde V(X)$ and $\tilde W(X)$, respectively, so that $a,b\leq n$ and $a+b=2d$. In case $d=-\infty$, one of $V(X)$ and $W(X)$ will be $0$ over $F$ and $R/\mathfrak p$ will be generated over ${k_L}$ by the coefficients of the other polynomial, and we see that $R/\mathfrak p$ has Krull dimension at most $n+1$.\\ 
				Now assume that $d\neq-\infty$. Write $\tilde U(X)=\tilde u_0+\dots+u_dX^d$ and so on. Let $\alpha_1,\dots,\alpha_d$ be the roots (with multiplicity) of $\tilde U(X)$ in some algebraic extension of $F$. Then the coefficients of $\frac1{\tilde u_d}\tilde U(X), \frac1{\tilde v_a}\tilde V(X),$ and $\frac1{\tilde w_b}\tilde W(X)$ are all in the ${k_L}$-algebra generated by the $\alpha_1,\dots,\alpha_d$. So we see that the ${k_L}$-algebra generated by $\alpha_1,\dots,\alpha_d$ and $\tilde v_a$ and $\tilde w_b$ contains the coefficients of $\tilde U(X),\tilde V(X),$ and $\tilde W(X)$, so that this ${k_L}$-algebra contains $R/\mathfrak p$.\\
				Therefore, $F$ has transcendence degree at most $d+2<n+2$ over ${k_L}$, so that the Krull dimension of $R/\mathfrak p$ is at most $n+1$ by Noether's Normalization Lemma, as we wished to prove.		

			\end{proof}
			\begin{remark}\label{explideal}
				The proof of Theorem \ref{mainmainthm} exhibits a certain ideal $I$ of $\gr k_L\llbracket H^{p^N}\rrbracket$ of which a certain power kills $\gr\overline M$. This suggests a generalization  beyond $K$ unramified of the annihilation condition defining the categories of mod $p$ representations of \cite{BHHMSConjs} and \cite{hu2024modprepresentationsquaternion} discussed in Section \ref{interm}.\\
				Making (\ref{casimirtobeexpanded}) explicit, via the identities in the proof of Lemma \ref{gammaialmostint}, one can make these ``Casimir ideals'' completely explicit.\\ For instance, if $K=\mathbf Q_p[\sqrt p]$, the ideal one finds (the radical of the one from the proof of Theorem \ref{mainmainthm}) is \begin{equation*}
					\left(\overline{\mathbf z}_0,\overline{\mathbf z}_1,\overline{\mathbf h}_1,\overline{\mathbf e}_0\overline{\mathbf f}_0,\overline{\mathbf e}_1\overline{\mathbf f}_1,\overline{\mathbf h}_0^2+4(\overline{\mathbf e}_1\overline{\mathbf f}_0+\overline{\mathbf e}_0\overline{\mathbf f}_1)\right)\subset k_L[\overline{\mathbf e}_i,\overline{\mathbf f}_i,\overline{\mathbf h}_i,\overline{\mathbf z}_i\colon i=0,1]=\gr k_L\llbracket H\rrbracket,
				\end{equation*} where $\overline{\mathbf x}_i$ denotes the principal symbol of $[\exp(\sqrt p^{i+1}p^2x)]-1$, for $x\in\{f,h,z\}$, while $\overline{\mathbf e}_i$ is the principal symbol of $[\exp(\sqrt p^ip^2x)]-1$.\\
				Finally, in the case $f>1$, let us remark that the ideal $I$ does not depend on the choices made regarding the embeddings $\rho_k\colon K\hookrightarrow L$.
			\end{remark}
			\subsection{Equality for principal series}\label{princseries}
			In the $\GL_2$-case we now show that the bound of Theorem \ref{mainmainthm} is optimal. Take a continuous character $\chi\colon B\to L^\times$, where $B\subset\GL_2K$ is the upper-triangular Borel. Then we define the \emph{principal series representation}\begin{equation*}
				\Ind_B^{\GL_2K}(\chi)\coloneqq\left\{f\colon\GL_2K\overset{\text{cont.}}{\to}L\colon f(bg)=\chi(b)^{-1}f(g),\forall b\in B, g\in\GL_2K\right\}.
			\end{equation*}
			By the Iwasawa decomposition $B\GL_2\mathcal O_K=\GL_2K$ its restriction to $\GL_2\mathcal O_K$ is isomorphic to $\Ind_{B\cap\GL_2\mathcal O_K}^{\GL_2\mathcal O_K}(\chi)\subset C(\GL_2\mathcal O_K,L)$, so that this is clearly an admissible $L$-Banach representation of $\GL_2K$. It has an infinitesimal character by (the proof of) \cite[Proposition 2.48]{abe2023irreducibilitypadicbanachprincipal}, so that its Gelfand-Kirillov dimension is at most $[K\colon\mathbf Q_p]$. This turns out to be an equality.
			\begin{proposition}
				The Gelfand-Kirillov dimension of $\Ind_B^{\GL_2K}(\chi)$ is $[K\colon\mathbf Q_p]$. In particular, the bound of Theorem \ref{mainmainthm} is optimal in the $\GL_2$-case.
			\end{proposition}
			\begin{proof}
				The restriction of $\Ind_B^{\GL_2K}(\chi)|_H$  naturally has $\Ind_{H\cap B}^H(\chi)$ as a subrepresentations. By \cite[Lemma A.8]{Gee_Newton_2022} combined with Proposition \ref{gradeq} the Gelfand-Kirillov dimension of $\Ind_{H\cap B}^H(\chi)$ is at most that of the principal series representation. So it suffices to show $\Ind_{H\cap B}^H(\chi)$ has Gelfand-Kirillov dimension $[K\colon\mathbf Q_p]$.\\
				The character $\chi$ induces a map $L\llbracket H\cap B\rrbracket\to L$ by evaluating the character on the group elements. Denote by $I_\chi$ the kernel of this map, which is a two-sided ideal. We have \begin{equation*}
					\Ind_{H\cap B}^H(\chi)^*=L\llbracket H\rrbracket/\left(L\llbracket H\rrbracket\cdot I_\chi\right)
				\end{equation*}
				By compactness the map exists integrally, $\mathcal O_L\llbracket H\cap B\rrbracket\to\mathcal O_L$, and we denote its kernel by $I_{\chi,\circ}$. Then $M_\circ\coloneqq\mathcal O_L\llbracket H\rrbracket/\left(\mathcal O_L\llbracket H\rrbracket\cdot I_{\chi,\circ}\right)$ gives $\Ind_{H\cap B}^H(\chi)^*$ when inverting $p$ and \begin{equation*}
					\gr\overline M=\gr k_L\llbracket H\rrbracket/\gr\left(k_L\llbracket H\rrbracket\cdot I_{\chi,\circ}\right),
				\end{equation*} where $k_L\llbracket H\rrbracket\cdot I_{\chi,\circ}$ is considered with the filtration induced from $k_L\llbracket H\rrbracket$.\\
				Since $\chi$ is continuous, $H$ pro-$p$, and $k_L^\times$ prime-to-$p$, the ideal $\gr\left(k_L\llbracket H\rrbracket\cdot I_{\chi,\circ}\right)$ of the polynomial $k_L$-algebra\begin{equation*}
					\gr k_L\llbracket H\rrbracket=k_L[\overline\sigma\left([x_{i,j}]-1\right)\colon0\leq i<f,0\leq j<e_K,x\in\{e,f,h,z\}]
				\end{equation*} (using the notation of (\ref{gl2orderedbasis})) is the ideal generated by the elements $\overline\sigma\left([x_{i,j}]-1\right)$ for $x\in\{e,h,z\}$, so that \begin{equation*}
					\gr\overline M\cong k_L[\overline\sigma\left([f_{i,j}]-1\right)\colon0\leq i<f,0\leq j<e_K]
				\end{equation*} has Krull dimension $[K\colon\mathbf Q_p]$.
			\end{proof}
			\subsection{Application to some representations from $p$-adic Langlands}
			We hope that our bound and proof method will have more interesting applications to the $p$-adic Langlands programme, but for now we only wish to point out a very straightforward application of our bound to certain representations occuring in $p$-adic Langlands.\\
			
			For $n\geq1$ an integer, \cite{Caraiani_Emerton_Gee_Geraghty_Paškūnas_Shin_2018} consider any continuous Galois representation $\overline r\colon\Gal_K\to\GL_nk_L$ and $R_{\overline r}^\square$ its framed deformation ring. Assuming $p\nmid2n$, they construct a patched module $\mathbb M_\infty$ by patching automorphic forms on definite unitary groups, which is a finitely generated module over $R_\infty\llbracket\GL_n\mathcal O_K\rrbracket$, where $R_\infty$ is a complete Noetherian local flat $R_{\overline r}^\square$-algebra with residue field $k_L$. The patched module $\mathbb M_\infty$ also carries an $R_\infty$-linear action of $\GL_nK$ which is compatible with the $\GL_n\mathcal O_K$-action.\\			
			For $y\in\SpecMax R_\infty[1/p]$ denote by $L_y$ the residue field and consider \begin{equation*}
				\Pi_y\coloneqq\Hom_{\mathcal O_{L_y}}^\text{cont}\left(\mathbb M_\infty\otimes_{R_\infty}\mathcal O_{L_y},L_y\right).
			\end{equation*} It is a (unitary) admissible Banach $L_y$-representation of $\GL_nK$ and it is considered a candidate for the $p$-adic Langlands correspondent to the Galois representation which is the image of $y$ in $\SpecMax R_{\overline r}^\square[1/p]$.\\
			In \cite[Thm.~9.27]{DPSinfchar} it is shown that all $\Pi_y$ have an infinitesimal character, so that we obtain:
			\begin{corollary}
				If $n=2$ and $y\in\SpecMax R_\infty[1/p]$, then the $L_y$-Banach representation $\Pi_y$ of $\GL_2K$ has Gelfand-Kirillov dimension $\leq[K\colon\mathbf Q_p]$.
			\end{corollary}
			\begin{remark}\label{patchremark} 
				\begin{enumerate}[label*=\alph*)]
					\item\label{patchremarka}  A much stronger result would be the same bound for the Gelfand-Kirillov dimension of the special fibre $\mathbb M_\infty/\mathfrak m_\infty$, where $\mathfrak m_\infty$ is the maximal ideal of $R_\infty$. As in \cite{BHHMS1} it would (at least when $R_\infty$ is formally smooth and of the right dimension) imply faithful flatness of $\mathbb M_\infty$ over $R_\infty$, which, for instance, implies that all $\Pi_y$ are non-zero.\\
					We unfortunately cannot apply Proposition \ref{gradeq} to obtain such bound, because we do not know whether any of the $\mathbb M_\infty\otimes_{R_\infty}\mathcal O_{L_y}$ is $p$-torsion free. Our methods for establishing the bound do not immediately generalize to the current situation of  \cite{DPSinfchar} of an ``infinitesimal character in families'', but we hope to be able to adapt our methods to this situation in the future.
					\item Finally, it is worth to note that \cite{DPSinfchar} provide conditions under which circumstances the Hecke eigenspaces of completed cohomology (at the place $p$) of locally symmetric spaces for a connected reductive group $G$ over $\mathbf Q$ have an infinitesimal character as $p$-adic Banach representation of $G(\mathbf Q_p)$ -- see their theorems 9.20 and 9.24. For certain Hecke eigenspaces in the case of modular curves or Shimura curves (sections \S9.7 and \S9.8 of \cite{DPSinfchar}) it is shown that these conditions are met. In these cases $G(\mathbf Q_p)$ is one of the groups we studied, so that we obtain the bound on the Gelfand-Kirillov dimensions there. 
				\end{enumerate}
			\end{remark}
			\bibliographystyle{alpha}
			\bibliography{Casimir} 
		\end{document}